\documentclass[a4paper, reqno, 10pt]{amsart}

\usepackage{amsmath,amssymb,amsthm}
\usepackage{mathrsfs}
\usepackage{natbib}
\usepackage{mathtools}

\usepackage[margin=1in]{geometry}
\allowdisplaybreaks[4]

\theoremstyle{plain}
\newtheorem{thm}{Theorem}[section]
\newtheorem{lem}[thm]{Lemma}

\newtheorem{cor}[thm]{Corollary}

\theoremstyle{definition}

\newtheorem{remk}[thm]{Remark}

\numberwithin{equation}{section}

\begin{document}
	
\title{
Lower bounds for the lifespan of solutions to the Euler--Korteweg equations}
		
\author[H. Kurabe]{Haruhiro Kurabe}
\address[H. Kurabe]{
	Graduate School of Mathematical Sciences, The University of Tokyo, 
	Tokyo 153-8914, Japan
}
\email{kurabe125@gmail.com}
	
\author[R. Takada]{Ryo Takada}
\address[R. Takada]{
	Graduate School of Mathematical Sciences, The University of Tokyo, 
	Tokyo 153-8914, Japan
}
\email{takada@ms.u-tokyo.ac.jp}

\author[T. Yoshizawa]{Tomoaki Yoshizawa}
\address[T. Yoshizawa]{
	Graduate School of Mathematical Sciences, The University of Tokyo, 
	Tokyo 153-8914, Japan
}
\email{tomoaki@ms.u-tokyo.ac.jp}

\begin{abstract}
We consider the initial value problem for the Euler--Korteweg equations in the class of irrotational flows. 
In every space dimension $d\geqslant2$, we establish a quantitative lower bound for the lifespan of solutions 
that explicitly exhibits its dependence on the Weber number, the Mach number, and the Sobolev norm of the initial perturbation.
In particular, this lower bound can be taken arbitrarily large
in a low-Weber-number regime under a suitable algebraic relation
between the Mach and Weber numbers.
\end{abstract}
	
\maketitle
	
\section{Introduction}

Let us consider the initial value problem for the compressible Euler--Korteweg equations
in $\mathbb{R}^d \, (d \geqslant 2)$:
\begin{equation}\label{EK_orig}
	\begin{cases}
		\partial_t \rho + \operatorname{div}(\rho u)=0
		&\quad t>0, \, x\in \mathbb{R}^d,
		\\
		\rho \left(
		\partial_t u +  u\cdot \nabla u
		\right) + \nabla \{ P(\rho) \}
		= \rho \nabla 
		\left\{
		K(\rho) \Delta \rho + \dfrac{1}{2}K^\prime(\rho)|\nabla \rho|^2
		\right\}
		&\quad t>0, \, x\in \mathbb{R}^d,
		\\
		\rho(0,x)=\rho_0(x),
		\quad
		u(0,x) = u_0(x)
		&\quad x \in \mathbb{R}^d.
	\end{cases}
\end{equation}
Here, the unknown functions $\rho=\rho(t,x)>0$ and $u=(u_1(t,x), \dots, u_d(t,x))$ denote
the density and the velocity of the fluid, respectively.
The pressure $P=P(\rho)$ is assumed to be a smooth function on $(0,\infty)$
satisfying $P^\prime(\rho_\ast)>0$ for some constant $\rho_\ast>0$.
The function $K=K(\rho)$ represents the capillary coefficient,
which is supposed to be smooth and positive on $(0,\infty)$.
The initial density and velocity are given by $\rho_0=\rho_0(x)>0$ and 
$u_0=(u_{0,1}(x),\dots, u_{0,d}(x))$, respectively.

In this paper, we are concerned with lower bounds for the lifespan of classical solutions
to \eqref{EK_orig}. More precisely, for the nondimensionalized system derived from \eqref{EK_orig},
we establish a quantitative lower bound that explicitly exhibits its dependence on
 the Weber number $\mathrm{We}$, the Mach number $\mathrm{Ma}$,
and the size of the initial perturbation in the Sobolev norms.
In particular, with $\kappa=1/\mathrm{We}$ and $c=1/\mathrm{Ma}$,
this lower bound guarantees arbitrarily long lifespans for sufficiently large $\kappa$,
provided that the algebraic growth rate of $c$ is close enough to that of $\sqrt{\kappa}$.

Before stating our main result precisely,
we shall review previous work on the Euler--Korteweg equations \eqref{EK_orig}.
For the thermodynamic derivation of the Korteweg stress tensor 
and the resulting inviscid, isothermal model, we refer to \citep{DS85,BGDDJ05}.
Concerning the local well-posedness of \eqref{EK_orig},
\citep{BGDD07} established the existence and uniqueness of local solutions
and the continuity of the solution map in every space dimension $d \geqslant 1$
for a general smooth positive capillary coefficient $K(\rho)$, a smooth pressure $P(\rho)$,
and initial data 
$(\rho_0-\rho_\ast, u_0) \in H^{s+2}(\mathbb{R}^d) \times H^{s+1}(\mathbb{R}^d) \ (s>d/2)$
such that $\rho_0$ is bounded away from zero.
A blow-up criterion for the corresponding local solutions was also obtained in \citep{BGDD07}.
As for the global well-posedness of \eqref{EK_orig},
\citep{AH18} treated the particular case $K(\rho)=\kappa_0/\rho$ with $\kappa_0>0$ and $P(\rho)= \rho^2/2$,
and established the existence of a unique global solution
for sufficiently small irrotational initial data in space dimensions $d\geqslant 3$.
For this choice of $K(\rho)$ and $P(\rho)$, the Euler--Korteweg system for irrotational flows corresponds to 
the Gross--Pitaevskii equation via the Madelung transform (see also \citep{CDS12}).
More generally, \citep{AH17} proved the existence of a unique global solution
in space dimensions $d\geqslant 3$ for a general smooth positive capillary coefficient $K(\rho)$,
a smooth pressure $P(\rho)$ satisfying the stability condition $P^\prime(\rho_\ast)>0$, 
and small irrotational initial data with sufficient regularity and suitable spatial decay.
The proof combines higher-order energy estimates with dispersive estimates for the linearized Gross--Pitaevskii equation
and employs the space-time resonance method in dimensions $d=3,4$ 
(see \citep{GNT06,GNT07,GNT09} for related results on the Gross--Pitaevskii equation).
The existence and stability of planar traveling waves have also been investigated in \citep{Aud17,BGDDJ05,BG13}.
For results on weak solutions, we refer to \citep{AM09,AM12,DFM15,GLT17}.

We next focus on known results concerning lower bounds for the lifespan of solutions to \eqref{EK_orig}.
For a general smooth positive capillary coefficient $K(\rho)$, a general smooth pressure $P(\rho)$,
and initial data $(\rho_0 - \rho_\ast, u_0) \in H^{s+2}(\mathbb{R}^d) \times H^{s+1}(\mathbb{R}^d) \ (s>d/2)$,
\citep{BGDD07} obtained the lower bound
 \[
 T \gtrsim \log\left(1+\frac{1}{Z_0+Z_0^2}\right)
 \]
in every space dimension $d\geqslant 1$,
where $Z_0 \coloneqq \| z_0 \|_{H^{s+1}} + \| z_0 \|_{H^{s-1}}^\beta$,
$z_0 \coloneqq u_0 + i \nabla \mathcal{L}(\rho_0)$, $\mathcal{L}^\prime(\rho) = \sqrt{K(\rho)/\rho}$,
and $\beta \geqslant 3$ is an exponent depending only on $s$ and $d$.
For the particular case $K(\rho) = 1/\rho$ and $P(\rho)=\rho^2$,
\citep{BDS10} considered small-amplitude, long-wave solutions to \eqref{EK_orig} of the form
 \[
 \rho(t,x) = 1 + \frac{\varepsilon}{\sqrt{2}} \tilde{\rho}(\varepsilon t, \varepsilon x),
 \quad
 u(t,x) = \varepsilon \tilde{u}(\varepsilon t, \varepsilon x)
 \]
with a small parameter $\varepsilon \in (0,1]$, 
and established the following lower bounds on the lifespan in the original time variable:
\[
T \gtrsim 
\begin{cases}
	\varepsilon^{-3} \| (\tilde{\rho}_0, \tilde{u}_0) \|_{H^{s+3} \times H^{s+2}}^{-2}
	&\quad (d\geqslant 4),
	\\
	\min\left\{
	\varepsilon^{-(2+\theta)} \| (\tilde{\rho}_0, \tilde{u}_0) \|_{H^{s+3} \times H^{s+2}}^{-(1+\theta)}, \
	 \varepsilon^{-4} \| (\tilde{\rho}_0, \tilde{u}_0) \|_{H^{s+3} \times H^{s+2}}^{-2}
	\right\}
	&\quad (d=3, \ 0<\theta<1),
	\\
	\min\left\{
	\varepsilon^{-\frac{7}{3}} \| (\tilde{\rho}_0, \tilde{u}_0) \|_{H^{s+3} \times H^{s+2}}^{-\frac{4}{3}}, \
	\varepsilon^{-(2+r)} \| (\tilde{\rho}_0, \tilde{u}_0) \|_{H^{s+3} \times H^{s+2}}^{-r}
	\right\}
	&\quad \left(d=2, \  \frac{2}{s}<r<2 \right)
\end{cases}
\]
for irrotational initial profiles $(\tilde{\rho}_0, \tilde{u}_0) \in H^{s+3}(\mathbb{R}^d) \times H^{s+2}(\mathbb{R}^d) \ (s>d/2)$
satisfying $\varepsilon  \| (\tilde{\rho}_0, \tilde{u}_0) \|_{H^{s+3} \times H^{s+2}} \lesssim 1$.
\citep{BGC18} also considered solutions to \eqref{EK_orig} of the form
 \[
 \rho(t,x)=\rho_\ast+\eta\tilde{\rho}(\varepsilon t,\varepsilon x),
 \quad
 u(t,x)=\eta\tilde{u}(\varepsilon t,\varepsilon x)
 \]
for a scaling parameter $\varepsilon>0$ and a small amplitude parameter $\eta \in (0,1]$, 
and derived the lower bound on the lifespan
\[
 T \gtrsim \frac{1}{\varepsilon \eta}
\]
for a smooth positive capillary coefficient $K(\rho)$, a smooth pressure $P(\rho)$
satisfying the stability condition $P^\prime(\rho_\ast)>0$,
and rescaled initial data $(\tilde{\rho}_0, \tilde{u}_0) \in H^{s+2}(\mathbb{R}^d) \times H^{s+1}(\mathbb{R}^d) \ (s>d/2)$
with $\| (\tilde{\rho}_0, \tilde{u}_0) \|_{H^{s+1}} + \varepsilon \| \tilde{\rho}_0 \|_{H^{s+2}} \lesssim 1$
in every space dimension $d\geqslant 1$.
In space dimensions $d\geqslant3$, \citep{Aud21} obtained lower bounds
on the lifespan in terms of the size of the solenoidal component of the initial velocity
for a general smooth positive capillary coefficient $K(\rho)$
and a smooth pressure $P(\rho)$ satisfying the stability condition $P^\prime(\rho_\ast)>0$.
Let $\mathbb{Q}\coloneqq-(-\Delta)^{-1}\nabla\operatorname{div}$
and $\mathbb{P}\coloneqq I-\mathbb{Q}$
denote the projections onto potential and solenoidal vector fields, respectively.
For $d\geqslant 5$, \citep{Aud21} proved that if 
$\| \rho_0-\rho_\ast \|_{H^{N+1}\cap W^{k+1,\frac{4}{3}}}$,
$\| u_0 \|_{H^N\cap W^{k,\frac{4}{3}}}$
and $\| \mathbb{P}u_0 \|_{W^{k,4}}$ are sufficiently small with suitably large $(N, k) \in \mathbb{N}^2$,
then the lifespan satisfies
\[
T\gtrsim\frac{1}{\| \mathbb{P}u_0 \|_{W^{k,4}}}.
\]
In dimensions $d=3,4$, the corresponding lower bound takes the form
\[
T\gtrsim
\frac{1}{
	\| \mathbb{P}u_0 \|_{W^{k,p}\cap W^{k,p^\prime}}
	+\| |x|\mathbb{P}u_0 \|_{L^2}
}
\]
for some exponent $p>2d/(d-2)$, provided that
$\| \rho_0-\rho_\ast \|_{H^{N+1}\cap W^{k+1,p^\prime}}$,
$\| |x|(\rho_0-\rho_\ast) \|_{L^2}$,
$\| u_0 \|_{H^N\cap W^{k,p^\prime}}$,
$\| |x|\mathbb{Q}u_0 \|_{L^2}$
and 
$\| \mathbb{P}u_0 \|_{W^{k,p}\cap W^{k,p^\prime}}
+\| |x|\mathbb{P}u_0 \|_{L^2}$
are sufficiently small with $(N, k) \in \mathbb{N}^2$ chosen large enough.
In particular, if $u_0$ is irrotational, then $\mathbb{P}u_0=0$
and the global well-posedness result of \citep{AH17} is recovered.

The present study is motivated in part by the lifespan estimate
for the three-dimensional compressible Euler equations with a Coriolis term
established in \citep{KPTW26}.
This estimate makes its dependence on the Mach number,
the Rossby number, and the size of the initial perturbation explicit.
Under the polytropic pressure law
$p(n)=n^\gamma/\gamma$ with $\gamma>1$,
the system considered in \citep{KPTW26} can be written as
\begin{equation}\label{EC}
	\begin{cases}
		\partial_t \rho + c \operatorname{div} u + u \cdot \nabla \rho + \alpha \rho \operatorname{div} u = 0, \\
		\partial_t u + \varepsilon^{-1} e_3 \times u + c \nabla\rho + u\cdot \nabla u + \alpha \rho \nabla\rho = 0,
		\\
		\rho(0,x) = \rho_0(x), \quad u(0,x) = u_0(x).
	\end{cases}
\end{equation}
Here, $\rho$ represents the perturbation of the sound-speed variable from a constant state,
$\alpha=(\gamma-1)/2$, and $e_3=(0,0,1)$ is the vertical unit vector.
The parameters $c=1/\mathrm{Ma}$ and $\varepsilon=\mathrm{Ro}$
denote the inverse Mach number and the Rossby number, respectively.
It was shown in \citep{KPTW26} that for any $q\in(2,\infty)$
and initial perturbation $(\rho_0,u_0)\in H^s(\mathbb{R}^3) \ (s>7/2)$,
the lifespan of solutions to \eqref{EC} satisfies the lower bound
\[
T\gtrsim
\varepsilon^{-\frac{1}{q-1}}
\min\left\{1,(c\varepsilon)^{\frac{3}{q-1}}\right\}
\|(\rho_0,u_0)\|_{H^s}^{-\frac{q}{q-1}}.
\]
In particular, this reduces to
$
T \gtrsim
\varepsilon^{-\frac{1}{q-1}}
\|(\rho_0,u_0)\|_{H^s}^{-\frac{q}{q-1}}$
in the regime $c\varepsilon\geqslant1$.
Thus, the lower bound increases with the rotation speed $\varepsilon^{-1}$.
For a related quantitative lifespan estimate in the setting of 
the two-dimensional inviscid Boussinesq system, see also \citep{JK26}.

In order to state our result more precisely, we first derive a non-dimensional form of the system \eqref{EK_orig}.
Let $L>0$ and $V>0$ denote the reference length and velocity scales, respectively.
We introduce the non-dimensional variables as
\[
t = \frac{L}{V} \tilde{t},
\quad 
x = L \tilde{x},
\quad
\rho(t,x) = \rho_\ast \tilde{\rho}(\tilde{t}, \tilde{x}),
\quad
u(t,x) = V \tilde{u}(\tilde{t}, \tilde{x})
\]
and
\[
P(\rho) = \rho_\ast P^\prime(\rho_\ast) \tilde{P}(\tilde{\rho}),
\quad
K(\rho) = K(\rho_\ast) \tilde{K}(\tilde{\rho}).
\]
With this scaling, the system \eqref{EK_orig} takes the following dimensionless form, 
where we omit the tildes for notational convenience:
\begin{equation}\label{EK_nond}
	\begin{cases}
		\partial_t \rho + \operatorname{div} (\rho u) =0, \\
		\rho\left(
		\partial_t u + u\cdot \nabla u
		\right) + c^2\nabla \{ P(\rho) \}
		= \kappa \rho \nabla \left\{
		K(\rho) \Delta \rho + \dfrac{1}{2}K^\prime(\rho) |\nabla \rho|^2
		\right\}
	\end{cases}
\end{equation}
with
\begin{equation*}
	P^\prime(1) = K(1) =1.
\end{equation*}
The non-dimensional parameters $c>0$ and $\kappa>0$ are given by
\begin{equation*}
	c=\frac{1}{\mathrm{Ma}} \coloneqq \frac{\sqrt{P^\prime(\rho_\ast)}}{V},
	\quad
	\kappa=\frac{1}{\mathrm{We}} \coloneqq \frac{\rho_\ast K(\rho_\ast)}{L^2V^2},
\end{equation*}
where $\mathrm{Ma}$ and $\mathrm{We}$ denote the Mach number and the Weber number associated with the above scaling, respectively.

Following \citep{BGDD07, AH17}, we reformulate the system \eqref{EK_nond} in terms of a density perturbation.
In what follows, we restrict our attention to irrotational flows and write
\[
u=\nabla\phi.
\]
We introduce the change of density variable
\begin{equation*}
	\mathcal{L}(\rho) 
	\coloneqq
	1+\int_1^{\rho} \sqrt{\frac{K(\tau)}{\tau}}\,d\tau,
	\quad \rho>0.
\end{equation*}
Since $\mathcal{L}^\prime(\rho)=\sqrt{K(\rho)/\rho}>0$,
the map $\mathcal{L}$ is a $C^\infty$-diffeomorphism from $(0,\infty)$ onto its image.
Then we define the rescaled density perturbation $\ell$ by
\[
\ell\coloneqq\frac{\mathcal{L}(\rho)-1}{\lambda},
\quad
\lambda\coloneqq \frac{1}{\max\{c,\sqrt{\kappa}\}}.
\]
Equivalently, $\rho=\mathcal{L}^{-1}(1+\lambda\ell)$.
We also introduce
\begin{equation*}
	a(\rho)\coloneqq\sqrt{\rho K(\rho)},
	\quad
	g(\rho)\coloneqq\int_1^\rho\frac{P'(\tau)}{\tau}\,d\tau,
	\quad 
	\tilde{a}\coloneqq a\circ\mathcal{L}^{-1},
	\quad
	\tilde{g}\coloneqq g\circ\mathcal{L}^{-1}.
\end{equation*}
Note that the normalizations $P'(1)=K(1)=\mathcal{L}(1)=1$ imply
$\tilde{a}(1)=\tilde{g}^\prime(1)=1$ and $\tilde{g}(1)=0$.
Under the above change of variables, a direct computation shows that 
the system \eqref{EK_nond} can be written in terms of $(\ell,\nabla\phi)$ as
\begin{equation}\label{EK_enp}
	\begin{cases}
		\partial_t \ell + \dfrac{1}{\lambda} \Delta \phi
		=- \nabla\phi \cdot \nabla \ell - \dfrac{1}{\lambda}
		\left\{
		\tilde{a}(1+\lambda \ell) -1
		\right\} \Delta \phi,
		\\[2mm]
		\partial_t (\nabla\phi)  
		+ (c^2 - \kappa \Delta) \lambda \nabla \ell
		\\
		\quad = 
		- \dfrac{1}{2}\nabla |\nabla \phi|^2
		+ \dfrac{\kappa \lambda^2}{2} \nabla |\nabla \ell|^2
		+ 
		\kappa \lambda \nabla 
		\left[
		\left\{
		\tilde{a}(1+\lambda \ell) -1
		\right\} \Delta \ell
		\right]
		- c^2 \nabla
		\left\{
		\tilde{g}(1+\lambda \ell) - \lambda \ell
		\right\},
		\\
		\ell(0,x) = \ell_0(x),
		\quad 
		\nabla \phi(0,x) = \nabla \phi_0(x).
	\end{cases}
\end{equation}
Here, $\ell_0
\coloneqq \lambda^{-1}\left\{ \mathcal{L}(\rho_0)-1 \right\}$
and $\nabla\phi_0\coloneqq u_0$.

We are now ready to state our main result for the reformulated system \eqref{EK_enp}.
In this paper, motivated in part by the lifespan estimate in \citep{KPTW26},
we shall derive a quantitative lower bound for the lifespan whose dependence on
the non-dimensional parameters 
$c$ and $\kappa$ and on the Sobolev norm of the initial perturbation
$(\ell_0,\nabla\phi_0)$ is explicit.
The result applies in every space dimension $d\geqslant2$
for the general capillary coefficient $K(\rho)$ and pressure $P(\rho)$ 
introduced above, under the stability condition $P^\prime(1)=1$.
In particular, it follows from our lower bound that
the lifespan can be made arbitrarily long by taking $\kappa$ sufficiently large, provided that
$c=\kappa^\alpha$ with $\frac{1}{2}-\frac{1}{4q}
<\alpha<
\frac{1}{2}+\frac{1}{2q}$ and $2<q<\infty$.

The precise statement of our main result is as follows.

\begin{thm}\label{main}
	Let $d \geqslant 2$ be an integer.
	Assume that 
	$K\in C^\infty((0,\infty);(0,\infty))$ and
	$P\in C^\infty((0,\infty);\mathbb{R})$ with $K(1)=P'(1)=1$.
	Let $q\in(2,\infty)$ if $d=2$, and let $q\in[2,\infty)$ if $d\geqslant 3$.
	Suppose that $s \in \mathbb{N}$ satisfies $s > d/2$.
	Then, there exist positive constants $\varepsilon=\varepsilon(d,s,K,P)$
	and $\delta=\delta(d,s,q,K,P)$ such that for every $c, \kappa>0$
	and $(\ell_0, \nabla \phi_0) \in H^{s+4}(\mathbb{R}^d) \times H^{s+3}(\mathbb{R}^d)$
	satisfying
	 \begin{equation}\label{small1}
	 	\|\nabla\phi_0\|_{H^{s}} + 
	 	\| \sqrt{c^2-\kappa \Delta} \, (\lambda \ell_0) \|_{H^{s}} \leqslant \varepsilon c,
	 \end{equation}
	 the system \eqref{EK_enp} admits a unique classical solution
	 $(\ell,\nabla\phi)$ in the class
	 \begin{equation*}
	 	 (\ell,\nabla\phi)
	 	\in
	 	C\left(
	 	[0,T];
	 	H^{s+4}(\mathbb{R}^d)
	 	\times
	 	H^{s+3}(\mathbb{R}^d)
	 	\right)
	 	\cap
	 	C^1\left(
	 	[0,T];
	 	H^{s+2}(\mathbb{R}^d)
	 	\times
	 	H^{s+1}(\mathbb{R}^d)
	 	\right),
	 \end{equation*}
	where the existence time $T$ can be chosen to satisfy
	\begin{equation}\label{lowbd}
		T \geqslant \delta
		\kappa^{\frac{q+1}{2(q-1)}} c^{\frac{2q}{q-1}}
		\left(
		\| (c^2-\kappa \Delta)^{\frac{3}{2}} \nabla \phi_0 \|_{H^s}
		+ \| (c^2-\kappa \Delta)^2 \lambda \ell_0 \|_{H^s}
		\right)^{-\frac{q}{q-1}}.
	\end{equation}
\end{thm}

\begin{remk}
	In the smallness condition \eqref{small1} and the lifespan estimate \eqref{lowbd}
	of Theorem \ref{main}, the parameters $c$ and $\kappa$ appear inside the norms
	of the initial perturbation $(\ell_0, \nabla \phi_0)$.
	We first give a sufficient smallness condition stated in terms of the Sobolev norm
	$\|(\ell_0,\nabla\phi_0)\|_{H^{s+1}\times H^s}$.
	Since $\lambda=\max\{c,\sqrt{\kappa}\}^{-1}$,
	it holds that
	\[
	\|\nabla\phi_0\|_{H^s}
	+
	\|\sqrt{c^2-\kappa\Delta}\,(\lambda\ell_0)\|_{H^s}
	\lesssim 
	\|\nabla\phi_0\|_{H^s}
	+
	\|\ell_0\|_{H^{s+1}}.
	\]
Therefore, taking  $\varepsilon$ smaller if necessary, we see that
the size condition \eqref{small1} is ensured by the stronger bound
	\[
	\|\nabla\phi_0\|_{H^s}
	+
	\|\ell_0\|_{H^{s+1}}
	\leqslant
	\varepsilon c.
	\]
We also remark that the admissible size of the initial perturbation in these Sobolev norms
can be taken larger as the inverse Mach number $c$ increases.
	
We next derive from \eqref{lowbd} a lower bound for the lifespan
in which the dependence on the initial perturbation is quantified by the Sobolev norm
$\|(\ell_0,\nabla\phi_0)\|_{H^{s+4}\times H^{s+3}}$.
Since
	\[
	\|(c^2-\kappa\Delta)^{\frac{3}{2}}\nabla\phi_0\|_{H^s}
	+
	\|(c^2-\kappa\Delta)^2(\lambda\ell_0)\|_{H^s}
	\lesssim 
	\max\{c,\sqrt{\kappa}\}^3
	\left(
	\|\nabla\phi_0\|_{H^{s+3}}
	+
	\|\ell_0\|_{H^{s+4}}
	\right),
	\]
it follows from \eqref{lowbd} that
	\begin{align*}
		T
		&\gtrsim
		\kappa^{\frac{1}{2(q-1)}}
		\min\left\{
		\frac{\sqrt{\kappa}}{c},
		\frac{c^2}{\kappa}
		\right\}^{\frac{q}{q-1}}
		\left(
		\|\nabla\phi_0\|_{H^{s+3}}
		+
		\|\ell_0\|_{H^{s+4}}
		\right)^{-\frac{q}{q-1}}
		\\
		&=
		\left(
		\|\nabla\phi_0\|_{H^{s+3}}
		+
		\|\ell_0\|_{H^{s+4}}
		\right)^{-\frac{q}{q-1}}
		\begin{cases}
			\kappa^{\frac{q+1}{2(q-1)}}c^{-\frac{q}{q-1}}
			&\quad (c\geqslant\sqrt{\kappa}),
			\\
			c^{\frac{2q}{q-1}}\kappa^{-\frac{2q-1}{2(q-1)}}
			&\quad (c\leqslant\sqrt{\kappa}).
		\end{cases}
	\end{align*}
Hence, for a fixed initial perturbation $(\ell_0, \nabla \phi_0)$ and $\kappa>1$,
the estimate \eqref{lowbd} yields $T\gg1$ whenever
\[
\kappa^{\frac{1}{2}+\frac{1}{2q}} \gg c \geqslant \sqrt{\kappa}
\quad \text{or}
\quad 
\sqrt{\kappa} \geqslant c \gg \kappa^{\frac{1}{2}-\frac{1}{4q}}.
\]
In particular, suppose that $c=\kappa^\alpha$ for some $\alpha \in \mathbb{R}$.
Then, we have for $\kappa>1$ 
\[
T \gtrsim
(\| \nabla \phi_0 \|_{H^{s+3}}
+ \| \ell_0 \|_{H^{s+4}})^{-\frac{q}{q-1}}
\begin{cases}
	\kappa^{\frac{q+1-2\alpha q}{2(q-1)}}
	&\quad (\alpha \geqslant 1/2),
	\\
	\kappa^{\frac{4\alpha q - (2q-1)}{2(q-1)}}
	&\quad (\alpha \leqslant 1/2).
\end{cases}
\]
It follows that this lower bound tends to infinity as $\kappa\to\infty$ provided that
\[
\frac{1}{2}-\frac{1}{4q}
<
\alpha
<
\frac{1}{2}+\frac{1}{2q}.
\]
In this regime, we also have $c=\kappa^\alpha \to \infty$ as
$\kappa\to\infty$ since $\alpha>0$.
Hence for any fixed initial perturbation
$(\ell_0, \nabla \phi_0) \in H^{s+4}(\mathbb{R}^d) \times H^{s+3}(\mathbb{R}^d)$,
the size condition \eqref{small1} holds for sufficiently large $\kappa$,
and arbitrarily long lifespans are guaranteed as $\kappa \to \infty$.
\end{remk}

We give a brief outline of the proof of Theorem \ref{main}.
Our argument adapts the bootstrap strategy used in \citep{KPTW26} to the present setting.
We combine higher-order energy estimates with the Strichartz estimates
for the linearized system in a bootstrap argument.
The Strichartz estimates yield bounds for $L^1_tL^\infty_x$-type space-time norms 
related to the continuation criterion,
whereas the energy estimates control the higher-order Sobolev norms of the solution.
To track the dependence on the non-dimensional parameters $c$ and $\kappa$,
we introduce the following parameter-weighted higher-order energy
and space-time norms:
\begin{align*}
	\mathcal{H}_s(T)
	&\coloneqq
	\|
	(c^2-\kappa\Delta)^{\frac{3}{2}}\nabla\phi
	\|_{L^\infty_T H^s_x}
	+
	\|
	(c^2-\kappa\Delta)^2(\lambda\ell)
	\|_{L^\infty_T H^s_x},
	\\
	\mathcal{B}(T)
	&\coloneqq
	\frac{c^2}{\sqrt{\kappa}}
	\sum_{j=0}^2
	\left(
	\frac{\sqrt{\kappa}}{c}
	\right)^j
	\|
	\nabla^j(\lambda\ell)
	\|_{L^1_TL^\infty_x}
	+
	\frac{c}{\sqrt{\kappa}}
	\sum_{j=0}^1
	\left(
	\frac{\sqrt{\kappa}}{c}
	\right)^j
	\|
	\nabla^j(\nabla\phi)
	\|_{L^1_TL^\infty_x}.
\end{align*}
For these norms, we establish the estimates
\begin{align*}
	\mathcal{H}_s(T)
	&\lesssim
	\mathcal{H}_s(0)
	\exp\left\{
	C\left(
	\mathcal{B}(T)
	+
	\lambda
	\int_0^T
	\|
	\nabla\phi(\tau)
	\|_{L^\infty}
	\|
	\nabla\ell(\tau)
	\|_{L^\infty}
	\,d\tau
	\right)
	\right\},
	\\
	\mathcal{B}(T)
	&\lesssim
	T^{1-\frac{1}{q}}
	\kappa^{-\frac{q+1}{2q}}
	c^{-2}
	\left\{
	\mathcal{H}_s(0)
	+
	\mathcal{B}(T)\mathcal{H}_s(T)
	\right\}.
\end{align*}
Using the smallness condition \eqref{small1} and the corresponding basic energy estimate, 
we bound the remaining integral term $\lambda
\int_0^T
\|
\nabla\phi(\tau)
\|_{L^\infty}
\|
\nabla\ell(\tau)
\|_{L^\infty}
\,d\tau$
in the exponential by a constant multiple of $\varepsilon\mathcal{B}(T)$.
Then, 
choosing $T$ to be of the order given by the right-hand side
of \eqref{lowbd} closes the bootstrap argument.
Moreover, the smallness condition \eqref{small1},
the corresponding basic energy estimate,
and the Sobolev embedding
$H^s(\mathbb{R}^d)\hookrightarrow L^\infty(\mathbb{R}^d)$
ensure that $\lambda\ell$ remains uniformly small
throughout the bootstrap interval.
Since $\rho=\mathcal{L}^{-1}(1+\lambda\ell)$,
the density stays bounded above and away from zero.
The preceding estimates and density bounds allow us to apply
the continuation criterion established in \citep{BGDD07}
and extend the solution up to the time $T$.

\medskip

This paper is organized as follows.
In Section 2, we reformulate the system \eqref{EK_enp}
as an extended system and derive higher-order energy estimates
using suitable gauge functions.
In Section 3, we recall dispersive and Strichartz estimates
for the linearized Gross--Pitaevskii equation
and apply them to derive an $L^1_tL^\infty_x$-type space-time estimate
for solutions to \eqref{EK_enp}.
In Section 4, we combine these estimates to close the bootstrap argument
and prove Theorem \ref{main}.

Throughout this paper, we use the following notation.
For $T>0$, $1\leqslant p\leqslant\infty$, and a Banach space $X(\mathbb{R}^d)$
of functions on $\mathbb{R}^d$, we set
\[
\|f\|_{L^p_TX_x}
\coloneqq
\|f\|_{L^p(0,T;X(\mathbb{R}^d))}.
\]
In the case $T=\infty$, we abbreviate
$L^p(0,\infty;X(\mathbb{R}^d))$ as $L^p_tX_x$.
The letter $C$ denotes a generic positive constant
whose value may change from line to line.
When its dependence on parameters needs to be specified,
we write $C=C(a,b,\ldots)$ to indicate that $C$ depends only
on the quantities appearing in parentheses.
For nonnegative quantities $A$ and $B$, the notation $A\lesssim B$
means that $A\leqslant CB$ for some constant $C>0$.
We write $A\gtrsim B$ if $B\lesssim A$, and $A\sim B$
if both $A\lesssim B$ and $B\lesssim A$.
All constants denoted by $C$, as well as those implicit in
$\lesssim$, $\gtrsim$, and $\sim$, are independent of
the non-dimensional parameters $c$ and $\kappa$.

\section{Energy Estimates}

In this section, we establish a priori energy estimates
for the solution to the reformulated system \eqref{EK_enp}, with explicit dependence on
the non-dimensional parameters $c$ and $\kappa$.
Following the energy method developed in \citep{BGDD07, AH17},
we formulate an extended system for the complex-valued unknown
$z=\nabla\phi+i\sqrt{\kappa}\lambda\nabla\ell$.
Then, by introducing suitable gauge functions, 
we cancel the derivative-loss contributions from the capillary terms.
The principal point for our purposes is to keep track of the dependence
on $c$ and $\kappa$, since it plays a crucial role in the bootstrap
argument leading to the lifespan bound \eqref{lowbd}.

For $s\in\mathbb{N}$ and $T>0$, we introduce the energy norms
\begin{align*}
	\mathcal{E}_s(T)
	&\coloneqq
	\|\nabla\phi\|_{L^\infty_T H^s_x}
	+
	\| \sqrt{c^2-\kappa\Delta}\,(\lambda\ell) \|_{L^\infty_T H^s_x},
	\\
	\mathcal{H}_s(T)
	&\coloneqq
	\|
	(c^2-\kappa\Delta)^{\frac{3}{2}}\nabla\phi
	\|_{L^\infty_T H^s_x}
	+
	\|
	(c^2-\kappa\Delta)^2(\lambda\ell)
	\|_{L^\infty_T H^s_x},
\end{align*}
and the space-time norm related to the blow-up criterion
\begin{equation*}
	\mathcal{B}(T)
	\coloneqq
	\frac{c^2}{\sqrt{\kappa}}\lambda
	\|\ell\|_{L^1_TL^\infty_x}
	+
	c\lambda
	\|\nabla\ell\|_{L^1_TL^\infty_x}
	+
	\sqrt{\kappa}\lambda
	\|\nabla^2\ell\|_{L^1_TL^\infty_x}
	+
	\frac{c}{\sqrt{\kappa}}
	\|\nabla\phi\|_{L^1_TL^\infty_x}
	+
	\|\nabla^2\phi\|_{L^1_TL^\infty_x}.
\end{equation*}
Here, $\mathcal{E}_s(T)$ denotes the basic energy norm,
whereas $\mathcal{H}_s(T)$ is the higher-order parameter-weighted norm.
We also set 
\begin{align*}
	\mathcal{E}_s(0)
	&\coloneqq
	\|\nabla\phi_0 \|_{H^s}
	+
	\| \sqrt{c^2-\kappa\Delta}\,(\lambda\ell_0) \|_{H^s},
	\\
	\mathcal{H}_s(0)
	&\coloneqq
	\|
	(c^2-\kappa\Delta)^{\frac{3}{2}}\nabla\phi_0
	\|_{H^s}
	+
	\|
	(c^2-\kappa\Delta)^2(\lambda\ell_0)
	\|_{H^s}.
\end{align*}
The following lemma provides energy estimates with explicit dependence
on $c$ and $\kappa$.

\begin{lem}\label{EnergyEst}
	Assume that $s \in \mathbb{N}$ satisfies $s>d/2$.
	Let $(\ell_0, \nabla \phi_0) \in H^{s+4}(\mathbb{R}^d) \times H^{s+3}(\mathbb{R}^d)$,
	and let $(\ell, \nabla \phi)$ be a classical solution to \eqref{EK_enp}
	in the class $C([0,T]; H^{s+4}(\mathbb{R}^d) \times H^{s+3}(\mathbb{R}^d))
	\cap C^1([0,T]; H^{s+2}(\mathbb{R}^d) \times H^{s+1}(\mathbb{R}^d))$ for some $T>0$.
	Suppose that there exist constants $0<m<1<M$ such that
	\[
	m \leqslant \rho(t,x)
	= \mathcal{L}^{-1}(1+\lambda \ell(t,x)) \leqslant M 
	\]
	for $(t,x) \in [0,T] \times \mathbb{R}^d$.
	Then, there exist positive constants $C_1=C_1(d,s,m,M,K)$,
	$C_2=C_2(d,s,m,M,K,P)$, $C_3=C_3(d,s,m,M,K)$ and
	$C_4=C_4(d,s,m,M,K,P)$ such that
	\begin{align}
		\mathcal{E}_s(t)
		&\leqslant C_1 \mathcal{E}_s(0)
		\exp\left\{ C_2 \left(
		\mathcal{B}(t) + \lambda \int_0^t \| \nabla \phi(\tau) \|_{L^\infty} \| \nabla \ell(\tau) \|_{L^\infty} \, d\tau
		\right) \right\},
		\label{EE1} \\
		\mathcal{H}_s(t)
		&\leqslant C_3 \mathcal{H}_s(0)
		\exp\left\{ C_4 \left(
		\mathcal{B}(t) + \lambda \int_0^t \| \nabla \phi(\tau) \|_{L^\infty} \| \nabla \ell(\tau) \|_{L^\infty} \, d\tau
		\right) \right\}
		\label{EE2}
	\end{align}
	for $t \in [0,T]$.
\end{lem}

\subsection{An Extended System}
The proof of Lemma \ref{EnergyEst} relies on an extended formulation
of the system \eqref{EK_enp}.
We first specify the vector-calculus notation used throughout this paper.
For vector fields $U,V\colon\mathbb{R}^d\to\mathbb{C}^d$, we set
\[
\left((\nabla U)\cdot V\right)_j
\coloneqq
\partial_j U\cdot V
=
\sum_{k=1}^d\partial_j U_k V_k
\]
for $j=1,2,\dots, d$.
With this convention, we have
\[
\nabla(U\cdot V)
=
(\nabla U)\cdot V
+
(\nabla V)\cdot U.
\]
Moreover, if $\operatorname{curl}V
\coloneqq \left(
\partial_k V_j - \partial_j V_k
\right)_{1\leqslant j,k \leqslant d}
=0$, it holds that
\[
(U\cdot\nabla)V=(\nabla V)\cdot U,
\quad
\nabla\operatorname{div}V=\Delta V.
\]

We now set
\[
u\coloneqq\nabla\phi,
\quad
w\coloneqq\nabla\ell,
\quad
z\coloneqq u+i\sqrt{\kappa}\lambda w.
\]
Note that $\operatorname{curl} u = \operatorname{curl} w=0$.
Rewriting the system \eqref{EK_enp} in terms of $u$ and $w$, and
applying $\nabla$ to the equation for $\ell$, we have
the following extended system for $(\ell, u, w)$:
 \begin{equation*}
 	\begin{cases}
 			\partial_t \ell + u \cdot w + \dfrac{1}{\lambda} a(\rho) \operatorname{div} u = 0,
 			\\
 			\partial_t u + u\cdot \nabla u 
 			- \kappa \lambda^2 w \cdot \nabla w
 			- \kappa \lambda \nabla (a(\rho) \operatorname{div} w)
 			+ c^2 \nabla (\tilde{g}(1+\lambda \ell))
 			=  0,
 			\\
 			\partial_t w + u \cdot \nabla w + w\cdot \nabla u + \dfrac{1}{\lambda} \nabla (a(\rho) \operatorname{div} u) = 0,
 	\end{cases}
 \end{equation*}
 where $\rho=\mathcal{L}^{-1}(1+\lambda\ell)$.
 Equivalently, multiplying the equation for $w$ by $i\sqrt{\kappa}\lambda$ and adding
 the resulting equation to the equation for $u$, we obtain
 \begin{equation}\label{Eq_z}
 	\begin{cases}
 		\partial_t \ell + u \cdot w + \dfrac{1}{\lambda} a(\rho) \operatorname{div} u = 0,
 		\\
 		\partial_t z + u \cdot \nabla z + i\sqrt{\kappa}\lambda w \cdot \nabla z
 		+ i \sqrt{\kappa} \nabla (a(\rho) \operatorname{div} z)
 		+ c^2 \lambda w
 		+ c^2 \lambda
 		\left\{
 		\tilde{g}^\prime(1+\lambda \ell)  -1
 		\right\} w =0.
 	\end{cases}
 \end{equation}
 
\subsection{Auxiliary Lemmas}
 
Before deriving the energy estimates \eqref{EE1} and \eqref{EE2}, we prepare several
product, commutator, and composition estimates that will be used in the
proof of Lemma \ref{EnergyEst} and throughout Sections 3 and 4.
We state them in the forms needed for the subsequent analysis.

The first auxiliary estimate concerns products of derivatives
(see \citep{KM81,Tay23}).
\begin{lem}\label{L_prod}
	Let $k \in \mathbb{N} \cup \{0\}$,
	and let $\alpha, \beta \in (\mathbb{N} \cup \{0\})^d$ satisfy
	$|\alpha|+|\beta|=k$.
	Then, there exists a positive constant $C=C(d,\alpha,\beta,k)$ such that
	\[
	\| (\partial^\alpha f) (\partial^\beta g) \|_{L^2}
	\leqslant C
	\left(
	\| f \|_{L^\infty} \| g \|_{\dot{H}^k}
	+ \| g \|_{L^\infty} \| f \|_{\dot{H}^k}
	\right)
	\]
	for $f, g \in L^\infty(\mathbb{R}^d) \cap H^k(\mathbb{R}^d)$.
\end{lem}

We next recall an estimate for the commutator of $\partial^\alpha$
with multiplication by $f$ (see \citep{KM81,KP88}).

\begin{lem}\label{L_comm}
	Let $k \in \mathbb{N}$, and let $\alpha \in (\mathbb{N} \cup \{0\})^d$
	satisfy $|\alpha|=k$.
	Then, there exists a positive constant $C=C(d, k)$ such that
	\[
	\| \partial^\alpha(fg) - f \partial^\alpha g \|_{L^2}
	\leqslant
	C\left(
	\| \nabla f \|_{L^\infty}\| g \|_{\dot{H}^{k-1}} + \| g \|_{L^\infty}\| f \|_{\dot{H}^k}
	\right)
	\]
	for 
	$f \in H^k(\mathbb{R}^d)$ with $\nabla f \in L^\infty(\mathbb{R}^d)$ and
	$g\in L^\infty(\mathbb{R}^d) \cap H^k(\mathbb{R}^d)$.
\end{lem}

Finally, we give two composition estimates in the Sobolev space
(see \citep{AG07, BGDD07}).
For $k\in\mathbb{N}$, we use the notation
$D^k f
\coloneqq
\bigl(\partial^\alpha f\bigr)_{|\alpha|=k}$.

\begin{lem}\label{L_comp1}
	Let $I\subset\mathbb{R}$ be an open interval, and let
	$J\Subset I$ be an interval such that $0 \in J$.
	Let $s\in\mathbb{N}$.
	Suppose that $F\in C^s(I)$ satisfies $F(0)=0$.
	Assume that $v\in H^s(\mathbb{R}^d)\cap L^\infty(\mathbb{R}^d)$ has values in $J$.
	Then, there exists a positive constant $C=C(d,s,I,J)$ such that
	\[
	\|F(v)\|_{H^s}
	\leqslant
	C
	\bigl(1+\|v\|_{L^\infty}\bigr)^{s-1}
	\|F'\|_{W^{s-1,\infty}(J)}
	\|v\|_{H^s}.
	\]
\end{lem}

\begin{lem}\label{L_comp2}
	Let $I\subset\mathbb{R}$ be an open interval,
	and let $J\Subset I$ be an interval.
	Let $k\in\mathbb{N}$ and $s\in\mathbb{N}\cup\{0\}$.
	Suppose that $F\in C^{s+k}(I)$ and that $v \in L^\infty(\mathbb{R}^d)$ has values in $J$
	and satisfies $D^k v \in H^s(\mathbb{R}^d)$.
	Then, there exists a positive constant $C=C(d,s,k,I,J)$ such that
	\[
	\bigl\|D^k(F(v))\bigr\|_{H^s}
	\leqslant
	C
	\bigl(1+\|v\|_{L^\infty}\bigr)^{s+k-1}
	\|F'\|_{W^{s+k-1,\infty}(J)}
	\|D^k v\|_{H^s}.
	\]
\end{lem}

\subsection{Proof of Lemma \ref{EnergyEst}}

The proof of Lemma \ref{EnergyEst} follows the arguments in \citep{BGDD07, AH17}.
The key ingredient is the use of suitably chosen gauge functions to cancel the
higher-order derivative terms that cause a loss of derivatives.
At the same time,
the present analysis requires precise control of the dependence on 
the non-dimensional parameters $c$ and $\kappa$.
We therefore provide a detailed proof, keeping track of the dependence
on $c$ and $\kappa$ throughout.

\begin{proof}[Proof of Lemma \ref{EnergyEst}]
We first consider derivatives of even order and derive the
$\dot{H}^{2n}$-estimate for $z$ with $n \in \mathbb{N}\cup \{0\}$.
To obtain the cancellation of the higher-order derivative terms, we
introduce the gauge function
\[
\varphi_n=\varphi_n(\rho)
\coloneqq
a(\rho)^n\sqrt{\rho}.
\]
Applying $\varphi_n\Delta^n$ to the equation for $z$ in
\eqref{Eq_z}, we derive an equation for $\varphi_n\Delta^n z$.
The main point is the expansion of the capillary term $i \sqrt{\kappa} \nabla (a \operatorname{div} z)$.
Since $z$ is curl-free, we have
$\nabla\operatorname{div}\Delta^{n-1}z = \Delta^n z$.
Then, the Leibniz rule yields
\[
\Delta^n(a\operatorname{div}z)
=
a\operatorname{div}\Delta^n z
+ 2n(\nabla a)\cdot\Delta^n z + R,
\]
where
\[
R
\coloneqq
\sum_{\substack{|\alpha|+|\beta|=2n \\ |\alpha|\geqslant2}}
C_{\alpha,\beta}
(\partial^\alpha a)
(\partial^\beta\operatorname{div}z).
\]
Hence we have
\begin{align*}
	i \sqrt{\kappa} \varphi_n \Delta^n \nabla (a \operatorname{div} z) 
	&= i \sqrt{\kappa}  \nabla ( a\operatorname{div} (\varphi_n \Delta^n z))
	\\
	&\quad + i \sqrt{\kappa}
	\left\{
	2n \varphi_n (\nabla a \cdot \nabla) \Delta^n z
	- a (\nabla \varphi_n \cdot \nabla) \Delta^n z
	- (\nabla \varphi_n) a \operatorname{div} \Delta^n z
	\right\}
	\\
	&\quad 
	- i \sqrt{\kappa}
	\left\{
	(\nabla a) (\nabla \varphi_n \cdot \Delta^n z)
	+ a (\Delta^n z \cdot \nabla) \nabla \varphi_n
	-2n \varphi_n (\Delta^n z \cdot \nabla) \nabla a
	\right\}
	\\
	&\quad + i \sqrt{\kappa} \varphi_n \nabla R.
\end{align*}
Also, the term $\varphi_n\Delta^n \left\{
(i\sqrt{\kappa}\lambda w\cdot\nabla)z
\right\}$ can be decomposed as
\begin{equation*}
	\varphi_n \Delta^n((i \sqrt{\kappa} \lambda w \cdot \nabla)z)
	= i\sqrt{\kappa} \varphi_n (\lambda w \cdot \nabla) \Delta^n z
	+ \varphi_n [\Delta^n, i\sqrt{\kappa} \lambda w\cdot \nabla] z.
\end{equation*}
Combining these identities, we see that $\varphi_n\Delta^n z$ satisfies
\begin{align}\label{Eq_2nz}
	&\partial_t (\varphi_n \Delta^n z)
	+ (u\cdot \nabla)(\varphi_n \Delta^n z)
	+  i \sqrt{\kappa}  \nabla ( a\operatorname{div} (\varphi_n \Delta^n z))
	+ c^2 \lambda \varphi_n \Delta^n w
	\notag \\
	&\quad 
	+ i \sqrt{\kappa}
	\left\{
	\varphi_n (\lambda w \cdot \nabla) \Delta^n z
	+ 2n \varphi_n (\nabla a \cdot \nabla) \Delta^n z
	- a (\nabla \varphi_n \cdot \nabla) \Delta^n z
	- (\nabla \varphi_n) a \operatorname{div} \Delta^n z
	\right\}
	\notag \\
	&\quad 
	+ \sum_{k=1}^6 \mathcal{N}_k = 0.
\end{align}
Here,
\begin{align*}
	\mathcal{N}_1 &\coloneqq  - (\partial_t \varphi_n + u \cdot \nabla \varphi_n) \Delta^n z,
	\notag \\
	\mathcal{N}_2 &\coloneqq \varphi_n [\Delta^n, u \cdot \nabla] z,
	\notag \\
	\mathcal{N}_3 &\coloneqq \varphi_n [\Delta^n, i\sqrt{\kappa} \lambda w\cdot \nabla] z,
	\notag \\
	\mathcal{N}_4 &\coloneqq  - i \sqrt{\kappa}
	\left\{
	(\nabla a) (\nabla \varphi_n \cdot \Delta^n z)
	+ a (\Delta^n z \cdot \nabla) \nabla \varphi_n
	-2n \varphi_n (\Delta^n z \cdot \nabla) \nabla a
	\right\},
	\notag \\
	\mathcal{N}_5 &\coloneqq i \sqrt{\kappa} \varphi_n \nabla R,
	\notag \\
	\mathcal{N}_6 &\coloneqq c^2 \lambda \varphi_n \Delta^n\left\{ (\tilde{g}^\prime(1+\lambda \ell) - 1)w \right\}.
\end{align*}
Taking the $L^2$-inner product of \eqref{Eq_2nz} with $\varphi_n\Delta^n z$
and then taking the real part, we have, after integration by parts,
\begin{align}\label{H2nz1}
	&\frac{1}{2} \frac{d}{dt}
	\| \varphi_n \Delta^n z \|_{L^2}^2
	- \frac{1}{2} \int_{\mathbb{R}^d}
	(\operatorname{div} u) |\varphi_n \Delta^n z|^2 \, dx
	+ c^2 \lambda \int_{\mathbb{R}^d} \varphi_n^2 \Delta^n w \cdot \Delta^n u \, dx
	\notag \\
	&\quad 
	+ \operatorname{Re}
	\bigg[
	i \sqrt{\kappa}
	\int_{\mathbb{R}^d}
	\big\{
	\varphi_n (\lambda w \cdot \nabla) \Delta^n z
	+ 2n \varphi_n (\nabla a \cdot \nabla) \Delta^n z
	\notag \\
	&\qquad\qquad\qquad\qquad
	- a (\nabla \varphi_n \cdot \nabla) \Delta^n z
	- (\nabla \varphi_n) a \operatorname{div} \Delta^n z
	\big\}
	\cdot \varphi_n \overline{\Delta^n z} \, dx
	\bigg]
	\notag \\
	&\quad = - \operatorname{Re}
	\left[
	\sum_{k=1}^6 \int_{\mathbb{R}^d} \mathcal{N}_k \cdot \varphi_n \overline{\Delta^n z} \, dx
	\right].
\end{align}
We now verify that the choice of $\varphi_n$ cancels the term in
\eqref{H2nz1} involving derivatives of $z$ of order $2n+1$.
Indeed, since $\lambda w=\frac{a}{\rho}\nabla\rho$, it follows that
\begin{align*}
	&\varphi_n (\lambda w \cdot \nabla) \Delta^n z
	+ 2n \varphi_n (\nabla a \cdot \nabla) \Delta^n z
	- a (\nabla \varphi_n \cdot \nabla) \Delta^n z
	- (\nabla \varphi_n) a \operatorname{div} \Delta^n z
	\\
	&\quad 
	= \varphi_n
	\left\{
	\left(
	\frac{a(\rho)}{\rho} + 2n a^\prime(\rho) - \frac{a  \varphi_n^\prime(\rho)}{\varphi_n(\rho)}
	\right) \nabla \rho \cdot \nabla \Delta^n z
	- \frac{a \varphi_n^\prime(\rho)}{\varphi_n(\rho)} (\nabla \rho) \operatorname{div} \Delta^n z
	\right\}
	\\
	&\quad 
	= \varphi_n
	\left\{
	\frac{a  \varphi_n^\prime(\rho)}{\varphi_n(\rho)} \nabla \rho \cdot \nabla \Delta^n z
	- \frac{a \varphi_n^\prime(\rho)}{\varphi_n(\rho)} (\nabla \rho) \operatorname{div} \Delta^n z
	\right\}
	\\
	&\quad 
	= \frac{1}{\varphi_n}
	\left\{
	(\nabla \Delta^n z) \cdot \nabla G(\rho)
	- (\operatorname{div} \Delta^n z) \nabla G(\rho)
	\right\},
\end{align*}
where $G(\rho) \coloneqq \int_1^{\rho}
a(\tau) \varphi_n^\prime(\tau) \varphi_n(\tau) 
\, d\tau$.
Since $\nabla G(\rho)$ is curl-free, applying
\citep[Lemma 3.1]{BGDD07} with
$Z= \Delta^n z$ and $W= \nabla G(\rho)$, we obtain
\begin{align}\label{ho}
	&\operatorname{Re}
	\bigg[
	i \sqrt{\kappa}
	\int_{\mathbb{R}^d}
	\big\{
	\varphi_n (\lambda w \cdot \nabla) \Delta^n z
	+ 2n \varphi_n (\nabla a \cdot \nabla) \Delta^n z
	\notag \\
	&\qquad\qquad\qquad
	- a (\nabla \varphi_n \cdot \nabla) \Delta^n z
	- (\nabla \varphi_n) a \operatorname{div} \Delta^n z
	\big\}
	\cdot \varphi_n \overline{\Delta^n z} \, dx
	\bigg]
	\notag \\
	&\quad =
	\operatorname{Re}
	\left[
	i \sqrt{\kappa}
	\int_{\mathbb{R}^d}
	\left\{
	(\nabla \Delta^n z) \cdot \nabla G(\rho)
	- (\operatorname{div} \Delta^n z) \nabla G(\rho)
	\right\}
	\cdot \overline{\Delta^n z} \, dx
	\right]
	\notag \\
	&\quad = - \operatorname{Im}
	\left[
	\sqrt{\kappa}
	\int_{\mathbb{R}^d}
	\left\{
	(\nabla \Delta^n z) \cdot \nabla G(\rho)
	- (\operatorname{div} \Delta^n z) \nabla G(\rho)
	\right\}
	\cdot \overline{\Delta^n z} \, dx
	\right]
	=0.
\end{align}
Therefore, the terms containing derivatives of order $2n+1$ vanish in the energy identity.
We next estimate the terms involving $\mathcal{N}_1,\ldots,\mathcal{N}_6$ on the right-hand side of \eqref{H2nz1}.
For $\mathcal{N}_1$, by the continuity equation and the definition of $\varphi_n$, we have
\begin{equation*}
	(\partial_t+u\cdot\nabla)\varphi_n
	=
	-\varphi_{n}^\prime \rho \operatorname{div}u
	=
	-\frac{\varphi_n}{2}
	\left(
	1+\frac{2n\rho a^\prime}{a}
	\right)
	\operatorname{div}u.
\end{equation*}
This gives 
\begin{equation}\label{N1-1}
	-\operatorname{Re}
	\int_{\mathbb{R}^d}
	\mathcal{N}_1\cdot\varphi_n\overline{\Delta^n z} \, dx
	=
	-\frac{1}{2}
	\int_{\mathbb{R}^d}
	(\operatorname{div}u)
	|\varphi_n\Delta^n z|^2\,dx
	-
	n\int_{\mathbb{R}^d}
	\frac{\rho a^\prime}{a}
	(\operatorname{div}u)
	|\varphi_n\Delta^n z|^2\,dx.
\end{equation}
The first term on the right-hand side cancels the transport term
$
-\frac{1}{2}
\int_{\mathbb{R}^d}
(\operatorname{div}u)
|\varphi_n\Delta^n z|^2\,dx
$
on the left-hand side of \eqref{H2nz1}.
Moreover, since $m\leqslant\rho\leqslant M$, we have
\begin{equation}\label{N1-2}
	\left|
	n\int_{\mathbb{R}^d}
	\frac{\rho a^\prime(\rho)}{a(\rho)}
	(\operatorname{div}u)
	|\varphi_n\Delta^n z|^2\,dx
	\right|
	\leqslant
	C(n,m,M,K)
	\|\operatorname{div}u\|_{L^\infty}
	\|\varphi_n\Delta^n z\|_{L^2}^2.
\end{equation}
Concerning $\mathcal{N}_2$, it follows from Lemma \ref{L_comm} that
\begin{align*}
	\left\|
	\left[ \Delta^n, \, u \cdot \nabla \right] z
	\right\|_{L^2}
	&\leqslant \sum_{|\alpha|=n} \frac{n!}{\alpha!}
	\left\|
	\partial^{2\alpha} (u \cdot \nabla z) - u \cdot \nabla \partial^{2\alpha} z
	\right\|_{L^2}
	\\
	&\leqslant
	C(d,n) 
	\left(
	\| \nabla u \|_{L^\infty} \| z \|_{\dot{H}^{2n}} + \| \nabla z \|_{L^\infty} \| u \|_{\dot{H}^{2n}}
	\right).
\end{align*}
Hence we have
\begin{align}\label{N2}
	&\left|
	\operatorname{Re}
	\int_{\mathbb{R}^d} 
	\mathcal{N}_2 
	\cdot \varphi_n \overline{\Delta^n z} \, dx
	\right|
	\notag \\
	&\quad \leqslant
	C(d,n,m,M,K)
	\left(
	\| \nabla u \|_{L^\infty} \| z \|_{\dot{H}^{2n}} + \| \nabla z \|_{L^\infty} \| u \|_{\dot{H}^{2n}}
	\right)
	\| \varphi_n \Delta^n z \|_{L^2}.
\end{align}
Similarly, Lemma \ref{L_comm} yields
\begin{align}\label{N3}
	&\left|
	\operatorname{Re}
	\int_{\mathbb{R}^d}
	\mathcal{N}_3\cdot \varphi_n \overline{\Delta^n z}\,dx
	\right|
	\notag \\
	&\quad
	\leqslant
	C(d,n,m,M,K)\sqrt{\kappa}\lambda
	\left(
	\|\nabla w\|_{L^\infty}\|z\|_{\dot{H}^{2n}}
	+
	\|\nabla z\|_{L^\infty}\|w\|_{\dot{H}^{2n}}
	\right)
	\|\varphi_n\Delta^n z\|_{L^2}.
\end{align}
For $\mathcal{N}_4$, we first observe that
\[
\int_{\mathbb{R}^d} (\nabla a)(\Delta^n z \cdot \nabla \varphi_n) \cdot \varphi_n \overline{\Delta^n z} \, dx
= \sum_{j,k=1}^d 
\int_{\mathbb{R}^d}
( \varphi_n a^\prime \varphi_n^\prime \partial_j \rho \partial_k \rho )
\Delta^n z_k \overline{\Delta^n z_j} \, dx.
\]
Since the matrix $( \varphi_n a^\prime \varphi_n^\prime \partial_j \rho \partial_k \rho )_{1\leqslant j,k \leqslant d}$
is real-symmetric, we have
\[
\operatorname{Re} 
\left[
- i \sqrt{\kappa} \int_{\mathbb{R}^d} (\nabla a)(\Delta^n z \cdot \nabla \varphi_n) 
\cdot \varphi_n \overline{\Delta^n z} \, dx
\right] = 0.
\]
The same argument applies to the remaining two terms in $\mathcal{N}_4$, and hence
\begin{equation}\label{N4}
	\operatorname{Re}
	\int_{\mathbb{R}^d} 
	\mathcal{N}_4
	\cdot \varphi_n \overline{\Delta^n z} \, dx = 0.
\end{equation}
As for $\mathcal{N}_5$, by the definition of $R$, we see that
\[
\nabla R
=
\sum_{\substack{|\alpha|+|\beta|=2n\\|\alpha|\geqslant2}}
C_{\alpha,\beta}
\left\{
(\partial^\alpha\nabla a)
(\partial^\beta\operatorname{div}z)
+
(\partial^\alpha a)
(\partial^\beta\nabla\operatorname{div}z)
\right\}.
\]
Hence, Lemma \ref{L_prod} gives
\begin{align*}
	\|\nabla R\|_{L^2}
	\leqslant
	C(d,n)
	\left(
	\|\nabla^2a\|_{L^\infty}
	\|\operatorname{div}z\|_{\dot{H}^{2n-1}}
	+
	\|\operatorname{div}z\|_{L^\infty}
	\|\nabla^2a\|_{\dot{H}^{2n-1}}
	\right).
\end{align*}
Note that $a=\tilde{a}(1+\lambda\ell)$ and
$\lambda \ell = \mathcal{L}(\rho)-1$.
The assumption $m\leqslant\rho\leqslant M$ implies 
\[
|\lambda \ell(t,x)| \leqslant \max\left\{ 1- \mathcal{L}(m), \mathcal{L}(M)-1 \right\}
\]
for $(t,x) \in [0,T] \times \mathbb{R}^d$.
Therefore, Lemma \ref{L_comp2} yields
\[
\|\nabla^2a\|_{\dot{H}^{2n-1}}
\leqslant
C(d,n,m,M,K)\lambda\|w\|_{\dot{H}^{2n}}.
\]
Moreover, the chain rule and the inequality
$
\|\nabla\ell\|_{L^\infty}^2
\lesssim
\|\ell\|_{L^\infty}
\|\nabla^2\ell\|_{L^\infty}
$
give
\begin{align*}
	\| \nabla^2 a \|_{L^\infty}
	&\leqslant
	C(m,M,K)
	\left(
	\lambda^2 \| \nabla \ell \|_{L^\infty}^2 + \lambda\| \nabla^2 \ell \|_{L^\infty}
	\right)
	\\
	&\leqslant
	C(m,M,K) \lambda \| \nabla w \|_{L^\infty}.
\end{align*}
Thus, it follows that
\begin{align}\label{N5}
	&\left|
	\operatorname{Re} 
	\int_{\mathbb{R}^d} 
	\mathcal{N}_5 \cdot \varphi_n \overline{\Delta^n z} \, dx
	\right|
	\notag \\
	&\quad \leqslant
	C(d,n,m,M,K)
	\sqrt{\kappa} \lambda
	\left(
	\| \nabla w \|_{L^\infty} \| z \|_{\dot{H}^{2n}}
	+ \| \operatorname{div} z \|_{L^\infty} \| w \|_{\dot{H}^{2n}}
	\right)
	\| \varphi_n \Delta^n z \|_{L^2}.
\end{align}
Finally, we estimate $\mathcal{N}_6$. By Lemma \ref{L_prod}, we have
\begin{align*}
	&\left|
	\operatorname{Re}
	\int_{\mathbb{R}^d}
	\mathcal{N}_6
	\cdot \varphi_n  \overline{\Delta^n z} \, dx
	\right|
	\\
	&\quad =
	\left|
	\int_{\mathbb{R}^d}
	c^2 \lambda \varphi_n \Delta^n\left\{ (\tilde{g}^\prime(1+\lambda \ell) - 1)w \right\}
	\cdot \varphi_n \Delta^n u \, dx
	\right|
	\\
	&\quad \leqslant 
	C(d,n,m,M,K)
	c^2 \lambda
	\big(
	\| w \|_{L^\infty}
	\| \tilde{g}^\prime(1+\lambda \ell) - 1 \|_{\dot{H}^{2n}}
	+ \| \tilde{g}^\prime(1+\lambda \ell) - 1 \|_{L^\infty}
	\| w \|_{\dot{H}^{2n}}
	\big)
	\| \varphi_n \Delta^n u \|_{L^2}.
\end{align*}
Since $\tilde{g}^\prime(1)=1$, the mean value theorem implies
\[
\left\|
\tilde{g}^\prime(1+\lambda\ell)-1
\right\|_{L^\infty}
\leqslant
C(m,M,K,P)\lambda\|\ell\|_{L^\infty}.
\]
Moreover, Lemma \ref{L_comp2} gives
\[
\left\|
\tilde{g}^\prime(1+\lambda\ell)-1
\right\|_{\dot{H}^{2n}}
\leqslant
C(d,n,m,M,K,P)
\lambda\|\ell\|_{\dot{H}^{2n}}.
\]
Hence we have
\begin{align}\label{N6}
	&\left|
	\operatorname{Re} \int_{\mathbb{R}^d}
	\mathcal{N}_6 \cdot \overline{\varphi_n \Delta^n z} \, dx
	\right|
	\notag \\
	&\quad \leqslant
	C(d,n,m,M,K,P)
	c^2 \lambda^2
	\left(
	\| \ell \|_{L^\infty} \| w \|_{\dot{H}^{2n}}
	+ \| w \|_{L^\infty} \| \ell \|_{\dot{H}^{2n}}
	\right) \| \varphi_n \Delta^n u \|_{L^2}.
\end{align}
Therefore, by \eqref{H2nz1}--\eqref{N6}, we obtain
\begin{align}\label{H2nz2}
	&\frac{1}{2} \frac{d}{dt}
	\| \varphi_n \Delta^n z \|_{L^2}^2
	+ c^2 \lambda \int_{\mathbb{R}^d} \varphi_n^2 \Delta^n w \cdot \Delta^n u \, dx
	\notag \\
	&\quad 
	\leqslant C(d,n,m,M,K,P)
	\big\{
	\left(
	\| \nabla u \|_{L^\infty} \| z \|_{\dot{H}^{2n}} + \| \nabla z \|_{L^\infty} \| u \|_{\dot{H}^{2n}}
	\right)
	\| \varphi_n \Delta^n z \|_{L^2}
	\notag \\
	&\qquad\qquad
	+ \sqrt{\kappa} \lambda 
	\left(
	\| \nabla w \|_{L^\infty} \| z \|_{\dot{H}^{2n}} + \| \nabla z \|_{L^\infty} \| w \|_{\dot{H}^{2n}}
	\right)
	\| \varphi_n \Delta^n z \|_{L^2}
	\notag \\
	&\qquad\qquad 
	+ c^2 \lambda^2
	\left(
	\| \ell \|_{L^\infty} \| w \|_{\dot{H}^{2n}}
	+ \| w \|_{L^\infty} \| \ell \|_{\dot{H}^{2n}}
	\right) \| \varphi_n \Delta^n u \|_{L^2}
	\big\}
	\notag \\
	&\quad \leqslant 
	C(d,n,m,M,K,P)
	\big\{
	\| \nabla z \|_{L^\infty}
	\| z \|_{\dot{H}^{2n}}\| \varphi_n \Delta^n z \|_{L^2}
	\notag \\
	&\qquad\qquad 
	+ c^2 \lambda^2
	\left(
	\| \ell \|_{L^\infty} \| w \|_{\dot{H}^{2n}}
	+ \| w \|_{L^\infty} \| \ell \|_{\dot{H}^{2n}}
	\right) \| \varphi_n \Delta^n u \|_{L^2}
	\big\}.
\end{align}

To cancel the quadratic term
$c^2\lambda
\int_{\mathbb{R}^d}\varphi_n^2 \Delta^n w\cdot \Delta^n u \, dx$
in \eqref{H2nz2}, we derive a weighted $\dot{H}^{2n}$-estimate for $\ell$ from the first equation in \eqref{Eq_z}.
Applying $\Delta^n$ to this equation, taking the $L^2$-inner product
of the resulting equation with
$c^2\lambda^2\frac{\varphi_n^2}{a}\Delta^n\ell$, 
we have
\[
\int_{\mathbb{R}^d}
\left\{
c^2\lambda^2 \frac{\varphi_n^2 \Delta^n \ell}{a}  \partial_t \Delta^n \ell
+ c^2\lambda^2 \frac{\varphi_n^2 \Delta^n \ell}{a} \Delta^n (u \cdot w)
+ c^2\lambda  \dfrac{\varphi_n^2 \Delta^n \ell}{a} \Delta^n (a \operatorname{div} u)
\right\} \, dx =0.
\]
For the third term in the left-hand side, 
using the decomposition
\[
\Delta^n (a \operatorname{div} u)=a \operatorname{div} \Delta^n u
+ \left[
\Delta^n, \, a
\right] \operatorname{div} u
\]
and integrating by parts, we have
\begin{align*}
	c^2\lambda
	\int_{\mathbb{R}^d}
	\frac{\varphi_n^2\Delta^n\ell}{a}
	\Delta^n(a\operatorname{div}u)\,dx
	&=
	-c^2\lambda
	\int_{\mathbb{R}^d}
	\varphi_n^2\Delta^n w\cdot\Delta^n u\,dx
	-c^2\lambda
	\int_{\mathbb{R}^d}
	\Delta^n\ell\,\nabla(\varphi_n^2)\cdot\Delta^n u\,dx
	\\
	&\quad
	+c^2\lambda
	\int_{\mathbb{R}^d}
	\frac{\varphi_n^2\Delta^n\ell}{a}
	[\Delta^n,a]\operatorname{div}u\,dx.
\end{align*}
Hence it holds that
\begin{align*}
	&\frac{c^2\lambda^2}{2} \frac{d}{dt}
	\left\|
	\frac{\varphi_n}{\sqrt{a}} \Delta^n \ell 
	\right\|_{L^2}^2
	-c^2\lambda
	\int_{\mathbb{R}^d}
	\varphi_n^2\Delta^n w\cdot\Delta^n u\,dx
	\\
	&\quad = - c^2 \lambda^2 \int_{\mathbb{R}^d} 
	\frac{\varphi_n^2 \Delta^n \ell}{a} \Delta^n (u \cdot w) \, dx
	- c^2\lambda
	\int_{\mathbb{R}^d}
	\frac{\varphi_n^2\Delta^n\ell}{a}
	[\Delta^n,a]\operatorname{div}u\,dx
	\\
	&\qquad 
	+ c^2\lambda
	\int_{\mathbb{R}^d}
	\Delta^n\ell\,\nabla(\varphi_n^2)\cdot\Delta^n u\,dx
	+ \frac{c^2\lambda^2}{2}
	\int_{\mathbb{R}^d} 
	\partial_t\left(
	\frac{\varphi_n^2}{a} 
	\right)
	|\Delta^n \ell|^2 \, dx.
\end{align*}
We estimate the four terms on the right-hand side separately.
By Lemma \ref{L_prod}, the first term is bounded by
\begin{align*}
	&\left|
	c^2\lambda^2
	\int_{\mathbb{R}^d}
	\frac{\varphi_n^2\Delta^n\ell}{a}
	\Delta^n(u\cdot w)\,dx
	\right|
	\\
	&\quad\leqslant
	C(d,n,m,M,K)c^2\lambda^2
	\left(
	\|u\|_{L^\infty}\|w\|_{\dot{H}^{2n}}
	+
	\|w\|_{L^\infty}\|u\|_{\dot{H}^{2n}}
	\right)
	\|\Delta^n\ell\|_{L^2}.
\end{align*}
For the second term, the Leibniz rule and Lemma \ref{L_prod} give
\begin{equation*}
	\|[\Delta^n,a]\operatorname{div}u\|_{L^2}
	\leqslant C(d,n)
	\left(
	\|\nabla a\|_{L^\infty}
	\|\operatorname{div}u\|_{\dot{H}^{2n-1}}
	+
	\|\operatorname{div}u\|_{L^\infty}
	\|\nabla a\|_{\dot{H}^{2n-1}}
	\right).
\end{equation*}
Since $a=\tilde{a}(1+\lambda\ell)$, the chain rule and
Lemma \ref{L_comp2} yield
\[
\|\nabla a\|_{L^\infty}
\leqslant
C(m,M,K)\lambda\|w\|_{L^\infty},
\quad 
\|\nabla a\|_{\dot{H}^{2n-1}}
\leqslant
C(d,n,m,M,K)\lambda\|\ell\|_{\dot{H}^{2n}}.
\]
Therefore, we have
\begin{align*}
	&\left|
	c^2 \lambda 
	\int_{\mathbb{R}^d}
	\dfrac{\varphi_n^2 \Delta^n \ell}{a} 
	\left[
	\Delta^n, \, a
	\right] \operatorname{div} u 
	\, dx
	\right|
	\\
	&\quad \leqslant C(d,n,m,M,K) 
	c^2 \lambda^2
	\left(
	\| w \|_{L^\infty}\| \operatorname{div} u \|_{\dot{H}^{2n-1}}
	+ \| \operatorname{div} u \|_{L^\infty}
	\| \ell \|_{\dot{H}^{2n}}
	\right)
	\| \Delta^n \ell \|_{L^2}.
\end{align*}
For the third term,
since $\nabla(\varphi_n^2)
=
2\lambda
\varphi_n\varphi_n^\prime
\frac{\rho}{a}w$, 
it holds that
\begin{equation*}
	\left|
	c^2\lambda
	\int_{\mathbb{R}^d}
	\Delta^n\ell\,
	\nabla(\varphi_n^2)\cdot\Delta^n u\,dx
	\right|
	\leqslant
	C(n,m,M,K)c^2\lambda^2
	\|w\|_{L^\infty}
	\|\Delta^n\ell\|_{L^2}
	\|\Delta^n u\|_{L^2}.
\end{equation*}
Finally, since
\[
\partial_t\left(\frac{\varphi_n^2}{a}\right)
=
- \left( \frac{\varphi_n^2}{a} \right)^{\!\prime}
\left(
\rho\operatorname{div}u+u\cdot\nabla\rho
\right),
\quad 
\nabla\rho = \lambda\frac{\rho}{a}w,
\]
we have
\begin{align*}
	&\left|
	\frac{c^2\lambda^2}{2}
	\int_{\mathbb{R}^d} 
	\partial_t\left(
	\frac{\varphi_n^2}{a} 
	\right)
	|\Delta^n \ell|^2 \, dx
	\right|
	\\
	&\quad \leqslant 
	C(n,m,M,K)
	\left\{
	c^2 \lambda^2 
	\| \operatorname{div} u \|_{L^\infty}
	\| \Delta^n \ell \|_{L^2}^2
	+ c^2 \lambda^3
	\| u \|_{L^\infty} \| w \|_{L^\infty}
	\| \Delta^n \ell \|_{L^2}^2
	\right\}.
\end{align*}
Combining these estimates gives
\begin{align}\label{H2nl}
	&\frac{c^2\lambda^2}{2} \frac{d}{dt}
	\left\|
	\frac{\varphi_n}{\sqrt{a}} \Delta^n \ell 
	\right\|_{L^2}^2
	-c^2\lambda
	\int_{\mathbb{R}^d}
	\varphi_n^2\Delta^n w\cdot\Delta^n u\,dx
	\notag \\
	&\quad \leqslant C(d,n,m,M,K)
	\big\{
	c^2 \lambda^3
	\| u \|_{L^\infty} \| w \|_{L^\infty}
	\| \Delta^n \ell \|_{L^2}^2
	\notag \\
	&\qquad + 
	c^2 \lambda^2
	\big(
	\| u \|_{L^\infty} \| w \|_{\dot{H}^{2n}}
	+ \| w \|_{L^\infty} \| u \|_{\dot{H}^{2n}}
	+ \| \operatorname{div} u \|_{L^\infty}
	\| \ell \|_{\dot{H}^{2n}}
	\big)\| \Delta^n \ell \|_{L^2}
	\big\}.
\end{align}

Therefore, adding \eqref{H2nz2} and \eqref{H2nl}, we obtain
\begin{align}\label{H2nest}
	&\frac{1}{2} \frac{d}{dt}
	\| \varphi_n \Delta^n z \|_{L^2}^2
	+  \frac{c^2\lambda^2}{2} \frac{d}{dt}
	\left\|
	\frac{\varphi_n}{\sqrt{a}} \Delta^n \ell 
	\right\|_{L^2}^2
	\notag \\
	&\quad \leqslant C(d,n,m,M,K,P) 
	\big\{ 
	\| \nabla z \|_{L^\infty}
	\| z \|_{\dot{H}^{2n}}
	\| \varphi_n \Delta^n z \|_{L^2}
	\notag \\
	&\qquad\quad 
	+ c^2 \lambda^2
	\left(
	\| \ell \|_{L^\infty} \| w \|_{\dot{H}^{2n}}
	+ \| w \|_{L^\infty} \| \ell \|_{\dot{H}^{2n}}
	\right) \| \varphi_n \Delta^n u \|_{L^2}
	\notag \\
	&\qquad\quad 
	+ c^2 \lambda^2
	\left(
	\| u \|_{L^\infty} \| w \|_{\dot{H}^{2n}}
	+ \| w \|_{L^\infty} \| u \|_{\dot{H}^{2n}}
	+ \| \operatorname{div} u \|_{L^\infty}
	\| \ell \|_{\dot{H}^{2n}}
	\right)\| \Delta^n \ell \|_{L^2}
	\notag \\
	&\qquad\quad 
	+ c^2 \lambda^3
	\| u \|_{L^\infty} \| w \|_{L^\infty}
	\| \Delta^n \ell \|_{L^2}^2
	\big\}.
\end{align}
Here, we define
 \begin{align*}
 	\dot{E}_{2n}(t)
 	&\coloneqq 
 	\| \varphi_n \Delta^n z(t) \|_{L^2}^2
 	+  c^2\lambda^2
 	\left\|
 	\frac{\varphi_n}{\sqrt{a}} \Delta^n \ell(t)
 	\right\|_{L^2}^2
 	\\
 	&= \| \varphi_n \Delta^n (\nabla \phi(t)) \|_{L^2}^2
 	+ \kappa \lambda^2
 	 \| \varphi_n \Delta^n (\nabla \ell(t)) \|_{L^2}^2
 	 + c^2 \lambda^2
 	\left\|
 	\frac{\varphi_n}{\sqrt{a}} \Delta^n \ell(t)
 	\right\|_{L^2}^2.
 \end{align*}
Since $m\leqslant\rho\leqslant M$ and $a(\rho)>0$, it holds that
\begin{equation}\label{E_equiv}
	\dot{E}_{2n}(t)
	\sim
	\|\nabla\phi(t)\|_{\dot{H}^{2n}}^2
	+\kappa\lambda^2\|\nabla\ell(t)\|_{\dot{H}^{2n}}^2
	+c^2\lambda^2\|\ell(t)\|_{\dot{H}^{2n}}^2.
\end{equation}
Here, the implicit constant depends on $d,n,m,M$ and $K$.
We estimate the terms on the right-hand side of \eqref{H2nest},
keeping the dependence on $c$ and $\kappa$ explicit:
\begin{align*}
	\| \nabla z \|_{L^\infty} \| z \|_{\dot{H}^{2n}} \| \varphi_n \Delta^n z \|_{L^2}
	&\lesssim 
	\left(
	\| \nabla^2 \phi \|_{L^\infty} + \sqrt{\kappa} \lambda \| \nabla^2 \ell \|_{L^\infty}
	\right) \dot{E}_{2n}(t),
	\\
	c^2 \lambda^2 \| \ell \|_{L^\infty} \| w \|_{\dot{H}^{2n}} \| \varphi_n \Delta^n u \|_{L^2}
	&\lesssim \frac{c^2}{\sqrt{\kappa}} \lambda \| \ell \|_{L^\infty} \dot{E}_{2n}(t),
	\\
	c^2\lambda^2 \| w \|_{L^\infty} \| \ell \|_{\dot{H}^{2n}} \| \varphi_n \Delta^n u \|_{L^2}
	&\lesssim c\lambda \| \nabla \ell \|_{L^\infty} \dot{E}_{2n}(t),
	\\
	c^2 \lambda^2 \|u \|_{L^\infty} \| w \|_{\dot{H}^{2n}} \| \Delta^n \ell \|_{L^2}
	&\lesssim \frac{c}{\sqrt{\kappa}} \| \nabla \phi \|_{L^\infty} \dot{E}_{2n}(t),
	\\
	c^2 \lambda^2
	\| w \|_{L^\infty} \| u \|_{\dot{H}^{2n}} \| \Delta^n \ell \|_{L^2}
	&\lesssim c \lambda \| \nabla \ell \|_{L^\infty} \dot{E}_{2n}(t),
	\\
	c^2 \lambda^2 \| \operatorname{div} u \|_{L^\infty}
	\| \ell \|_{\dot{H}^{2n}} \| \Delta^n \ell \|_{L^2}
	&\lesssim \| \nabla^2 \phi \|_{L^\infty} \dot{E}_{2n}(t),
	\\
	c^2 \lambda^3 
	\| u \|_{L^\infty} \| w \|_{L^\infty} \| \Delta^n \ell \|_{L^2}^2
	&\lesssim 
	\lambda
	\| \nabla \phi \|_{L^\infty} \| \nabla \ell \|_{L^\infty} \dot{E}_{2n}(t).
\end{align*}
Here, the implicit constants in the preceding estimates depend only on
$d,n,m,M$, and $K$.
Hence it follows from \eqref{H2nest} that
\begin{align*}
	\frac{1}{2} \frac{d}{dt} \dot{E}_{2n}(t)
	&\leqslant C^\prime 
	\bigg\{
	\left(
	\frac{c^2}{\sqrt{\kappa}} \lambda \| \ell \|_{L^\infty}
	+ c \lambda \| \nabla \ell \|_{L^\infty}
	+ \sqrt{\kappa} \lambda \| \nabla^2 \ell \|_{L^\infty}
	\right)
	\\
	&\qquad\quad  +
	\left(
	\frac{c}{\sqrt{\kappa}} \| \nabla \phi \|_{L^\infty}
	+ \| \nabla^2 \phi \|_{L^\infty}
	\right)
	+ \lambda \| \nabla \phi \|_{L^\infty} \| \nabla \ell \|_{L^\infty}
	\bigg\} \dot{E}_{2n}(t)
\end{align*}
with $C^\prime=C^\prime(d,n,m,M,K,P)>0$.
Applying the Gronwall inequality and using the equivalence \eqref{E_equiv},
we obtain the $\dot{H}^{2n}$-estimate
 \begin{align}\label{H2nest2}
 	&\|\nabla\phi(t)\|_{\dot{H}^{2n}}^2
 	+\kappa\lambda^2\|\nabla\ell(t)\|_{\dot{H}^{2n}}^2
 	+c^2\lambda^2\|\ell(t)\|_{\dot{H}^{2n}}^2
 	\notag \\
 	&\quad \leqslant
 	C\left(
 	\|\nabla\phi_0\|_{\dot{H}^{2n}}^2
 	+\kappa\lambda^2\|\nabla\ell_0\|_{\dot{H}^{2n}}^2
 	+c^2\lambda^2\|\ell_0\|_{\dot{H}^{2n}}^2
 	\right)
 	\exp\left\{ C^\prime \left(
 	\mathcal{B}(t) + \lambda \int_0^t \| \nabla \phi(\tau) \|_{L^\infty} \| \nabla \ell(\tau) \|_{L^\infty} \, d\tau
 	\right) \right\}
 \end{align}
for $0 \leqslant t \leqslant T$ with $C=C(d,n,m,M,K)>0$.

We next consider derivatives of odd order and derive the $\dot{H}^{2n+1}$-estimate for $z$.
The proof proceeds in essentially the same way as in the even-order case. 
The only point requiring a modification is the cancellation of the terms containing derivatives of order $2n+2$.
We therefore give the details only for this cancellation.

We introduce the gauge function
\[
\psi_n=\psi_n(\rho)
\coloneqq
a(\rho)^{\frac{2n+1}{2}}\sqrt{\rho}.
\]
Let $j \in \{1,2,\ldots, d\}$.
Applying $\psi_n\partial_j \Delta^n$ to the equation for $z$ in
\eqref{Eq_z}, we have
\begin{align}\label{Eq_2n1z}
	&\partial_t (\psi_n \partial_j \Delta^n z)
	+ (u\cdot \nabla)(\psi_n \partial_j \Delta^n z)
	+  i \sqrt{\kappa}  \nabla ( a\operatorname{div} (\psi_n \partial_j \Delta^n z))
	+ c^2 \lambda \psi_n \partial_j \Delta^n w
	\notag \\
	&\quad 
	+ i \sqrt{\kappa}
	\big\{
	\psi_n (\lambda w \cdot \nabla) \partial_j \Delta^n z 
	+ 2 n \psi_n (\nabla a \cdot \nabla) \partial_j \Delta^n z
	- a (\nabla \psi_n \cdot \nabla) \partial_j \Delta^n z
	- (\nabla \psi_n) a \operatorname{div} \partial_j \Delta^n z
	\notag \\
	&\qquad\qquad\quad 
	+ \psi_n (\partial_j a) \Delta^{n+1} z
	\big\}
	\notag \\
	&\quad 
	+ \sum_{k=1}^6 \mathcal{N}^\prime_{k,j} = 0.
\end{align}
Here,
\begin{align*}
	\mathcal{N}_{1,j}^\prime &\coloneqq  - (\partial_t \psi_n + u \cdot \nabla \psi_n) \partial_j \Delta^n z,
	\notag \\
	\mathcal{N}_{2,j}^\prime &\coloneqq \psi_n [\partial_j \Delta^n, u \cdot \nabla] z,
	\notag \\
	\mathcal{N}_{3,j}^\prime &\coloneqq \psi_n [\partial_j \Delta^n, i\sqrt{\kappa} \lambda w\cdot \nabla] z,
	\notag \\
	\mathcal{N}_{4,j}^\prime &\coloneqq  - i \sqrt{\kappa}
	\left\{
	(\nabla a) (\nabla \psi_n \cdot \partial_j \Delta^n z)
	+ a (\partial_j \Delta^n z \cdot \nabla) \nabla \psi_n
	-2n \psi_n (\partial_j \Delta^n z \cdot \nabla) \nabla a
	- \psi_n (\nabla \partial_j a) \operatorname{div} \Delta^n z
	\right\},
	\notag \\
	\mathcal{N}_{5,j}^\prime &\coloneqq i \sqrt{\kappa} \psi_n \nabla R^\prime_j,
	\notag \\
	\mathcal{N}_{6,j}^\prime &\coloneqq c^2 \lambda \psi_n \partial_j \Delta^n\left\{ (\tilde{g}^\prime(1+\lambda \ell) - 1)w \right\},
	\notag \\
	R^\prime_j &\coloneqq 
	2n (\nabla \partial_j a) \cdot \Delta^n z 
	+ 
	\sum_{\substack{|\alpha|+|\beta|=2n \\ |\alpha|\geqslant2}}
	C_{\alpha, \beta} 
	\left\{
	(\partial_j \partial^\alpha a) (\partial^\beta \operatorname{div} z)
	+ (\partial^\alpha a) (\partial_j \partial^\beta \operatorname{div} z)
	\right\}.
\end{align*}
The terms in braces in \eqref{Eq_2n1z} contain derivatives of $z$ of order $2n+2$.
Taking the $L^2$-inner product of \eqref{Eq_2n1z} with
$\psi_n\partial_j\Delta^n z$, summing over
$j\in\{1,\ldots,d\}$, and then taking the real part, we see that their contribution is
\begin{align*}
	&\operatorname{Re}
	\bigg[
	i \sqrt{\kappa} 
	\sum_{j=1}^d 
	\int_{\mathbb{R}^d}
	\big\{
	\psi_n (\lambda w \cdot \nabla) \partial_j \Delta^n z 
	+ 2 n \psi_n (\nabla a \cdot \nabla)\partial_j \Delta^n z
	\\
	&\qquad\qquad\qquad\qquad
	- a (\nabla \psi_n \cdot \nabla) \partial_j \Delta^n z
	- (\nabla \psi_n) a \operatorname{div} \partial_j \Delta^n z
	\big\}
	\cdot \psi_n \overline{\partial_j \Delta^n z} \, dx
	\bigg]
	\\
	&\quad + \operatorname{Re}
	\left[
	i \sqrt{\kappa} \sum_{j=1}^d \int_{\mathbb{R}^d}
	\psi_n (\partial_j a) \Delta^{n+1} z \cdot \psi_n \overline{\partial_j \Delta^n z} \, dx 
	\right].
\end{align*}
For the last term, integration by parts gives
\begin{align*}
	&\sum_{j=1}^d \int_{\mathbb{R}^d}
	\psi_n (\partial_j a) \Delta^{n+1} z \cdot \psi_n \overline{\partial_j \Delta^n z} \, dx 
	= \sum_{j,k=1}^d
	\int_{\mathbb{R}^d}
	\psi_n^2 (\partial_j a) \partial_k^2 \Delta^{n} z \cdot \overline{\partial_j \Delta^n z} \, dx 
	\\
	&\quad = - \sum_{j,k=1}^d
	\int_{\mathbb{R}^d}
	\partial_k (\psi_n^2 (\partial_j a)) \partial_k \Delta^{n} z \cdot \overline{\partial_j \Delta^n z} \, dx 
	- \sum_{j=1}^d
	\overline{
		\int_{\mathbb{R}^d}
		\psi_n (\nabla a \cdot \nabla) \partial_j \Delta^n z \cdot 
		\psi_n \overline{\partial_j \Delta^{n} z}  \, dx}.
\end{align*}
Since $\{\partial_k (\psi_n^2 (\partial_j a))\}_{1\leqslant j,k \leqslant d}$ is real-symmetric, we have
\begin{equation*}
	\operatorname{Re}
	\left[
	i \sqrt{\kappa} \sum_{j=1}^d \int_{\mathbb{R}^d}
	\psi_n (\partial_j a) \Delta^{n+1} z \cdot \psi_n \overline{\partial_j \Delta^n z} \, dx 
	\right] 
	= \operatorname{Re}
	\left[
	i \sqrt{\kappa}\sum_{j=1}^d
	\int_{\mathbb{R}^d}
	\psi_n (\nabla a \cdot \nabla) \partial_j \Delta^n z \cdot \psi_n \overline{\partial_j \Delta^{n} z}  \, dx
	\right].
\end{equation*}
Hence it follows that the total contribution of the terms containing
derivatives of order $2n+2$ reduces to
\begin{align*}
	&\operatorname{Re}
	\bigg[
	i\sqrt{\kappa}
	\sum_{j=1}^d
	\int_{\mathbb{R}^d}
	\big\{
	\psi_n(\lambda w\cdot\nabla)\partial_j\Delta^n z
	+(2n+1)\psi_n(\nabla a\cdot\nabla)\partial_j\Delta^n z
	\\
	&\qquad\qquad\qquad\qquad
	-a(\nabla\psi_n\cdot\nabla)\partial_j\Delta^n z
	-(\nabla\psi_n)a\operatorname{div}\partial_j\Delta^n z
	\big\}
	\cdot\psi_n\overline{\partial_j\Delta^n z}\,dx
	\bigg].
\end{align*}
This expression has exactly the same structure as the highest-order
contribution in \eqref{H2nz1}, with $\varphi_n$ and $2n$ replaced by
$\psi_n$ and $2n+1$, respectively.
Thus, the same argument as in \eqref{ho} shows that the above contribution vanishes.
The remaining terms are estimated in the same manner as in the even-order case,
using the same product, commutator, and composition estimates.
After summing over $j\in\{1,\ldots,d\}$, we obtain
\begin{align}\label{H2n1est}
	&\frac{1}{2} \frac{d}{dt}
	\| \psi_n \nabla \Delta^n z \|_{L^2}^2
	+  \frac{c^2\lambda^2}{2} \frac{d}{dt}
	\left\|
	\frac{\psi_n}{\sqrt{a}} \nabla \Delta^n \ell 
	\right\|_{L^2}^2
	\notag \\
	&\quad \leqslant C(d,n,m,M,K,P) 
	\big\{ 
	\| \nabla z \|_{L^\infty}
	\| z \|_{\dot{H}^{2n+1}}
	\| \psi_n \nabla \Delta^n z \|_{L^2}
	\notag \\
	&\qquad\quad 
	+ c^2 \lambda^2
	\left(
	\| \ell \|_{L^\infty} \| w \|_{\dot{H}^{2n+1}}
	+ \| w \|_{L^\infty} \| \ell \|_{\dot{H}^{2n+1}}
	\right) \| \psi_n \nabla \Delta^n u \|_{L^2}
	\notag \\
	&\qquad\quad 
	+ c^2 \lambda^2
	\left(
	\| u \|_{L^\infty} \| w \|_{\dot{H}^{2n+1}}
	+ \| w \|_{L^\infty} \| u \|_{\dot{H}^{2n+1}}
	+ \| \operatorname{div} u \|_{L^\infty}
	\| \ell \|_{\dot{H}^{2n+1}}
	\right)\| \nabla \Delta^n \ell \|_{L^2}
	\notag \\
	&\qquad\quad 
	+ c^2 \lambda^3
	\| u \|_{L^\infty} \| w \|_{L^\infty}
	\| \nabla \Delta^n \ell \|_{L^2}^2
	\big\}.
\end{align}
The estimate \eqref{H2n1est} has the same form as \eqref{H2nest}, 
with $\varphi_n$, $\|\cdot\|_{\dot{H}^{2n}}$, and $\Delta^n$ replaced by
$\psi_n$, $\|\cdot\|_{\dot{H}^{2n+1}}$, and $\nabla\Delta^n$, respectively.
Therefore, the same argument shows that \eqref{H2nest2} remains valid
with $2n$ replaced by $2n+1$.
Combining the even- and odd-order cases, we conclude that, 
for every $n\in\mathbb{N}\cup\{0\}$ satisfying $n\leqslant s+3$, 
the following $\dot{H}^n$-estimate holds
\begin{align}\label{Hnest}
	&\|\nabla\phi(t)\|_{\dot{H}^{n}}^2
	+\kappa\lambda^2\|\nabla\ell(t)\|_{\dot{H}^{n}}^2
	+c^2\lambda^2\|\ell(t)\|_{\dot{H}^{n}}^2
	\notag \\
	&\quad \leqslant
	C\left(
	\|\nabla\phi_0\|_{\dot{H}^{n}}^2
	+\kappa\lambda^2\|\nabla\ell_0\|_{\dot{H}^{n}}^2
	+c^2\lambda^2\|\ell_0\|_{\dot{H}^{n}}^2
	\right)
	\exp\left\{ C^\prime \left(
	\mathcal{B}(t) + \lambda \int_0^t \| \nabla \phi(\tau) \|_{L^\infty} \| \nabla \ell(\tau) \|_{L^\infty} \, d\tau
	\right) \right\}
\end{align}
for $0 \leqslant t \leqslant T$, where $C=C(d,n,m,M,K)>0$
and $C^\prime=C^\prime(d,n,m,M,K,P)>0$.

Here, we remark that
 \[
 \kappa\lambda^2\|\nabla\ell(t)\|_{\dot{H}^{n}}^2
 +c^2\lambda^2\|\ell(t)\|_{\dot{H}^{n}}^2
 \sim \| \sqrt{c^2 - \kappa \Delta} \, (\lambda\ell(t)) \|_{\dot{H}^n}^2.
 \]
Thus, summing \eqref{Hnest} over $n=0,1,\ldots,s$, we obtain \eqref{EE1}.

For \eqref{EE2}, let $n\in\{0,1,\ldots,s\}$.
We multiply \eqref{Hnest} at order $n$ by $c^6$ and \eqref{Hnest} at order $n+3$ by $\kappa^3$,
and then add the resulting estimates.
Since $c^6+\kappa^3|\xi|^6
\sim
(c^2+\kappa|\xi|^2)^3$, we obtain
\begin{align*}
	&\|
	(c^2-\kappa\Delta)^{\frac{3}{2}}\nabla\phi(t)
	\|_{\dot{H}^n}^2
	+
	\|
	(c^2-\kappa\Delta)^2 \lambda \ell(t)
	\|_{\dot{H}^n}^2
	\\
	&\leqslant
	C
	\left(
	\|
	(c^2-\kappa\Delta)^{\frac{3}{2}}\nabla\phi_0
	\|_{\dot{H}^n}^2
	+
	\|
	(c^2-\kappa\Delta)^2 \lambda \ell_0
	\|_{\dot{H}^n}^2
	\right)
	\exp\left[
	C^\prime \left\{
	\mathcal{B}(t)
	+\lambda\int_0^t
	\|\nabla\phi(\tau)\|_{L^\infty}
	\|\nabla\ell(\tau)\|_{L^\infty}\,d\tau
	\right\}
	\right],
\end{align*}
where $C=C(d,s,m,M,K)>0$ and $C^\prime=C^\prime(d,s,m,M,K,P)>0$.
Summing this estimate over $n=0,1,\ldots,s$ yields
\eqref{EE2}. This completes the proof of Lemma \ref{EnergyEst}.
\end{proof}

\section{Space-Time Estimates}

In this section, following the approach of \citep{AH17,GNT06}, 
we introduce the dispersive unknown 
and rewrite the system \eqref{EK_enp} as a system of integral equations. 
Then we use the Strichartz estimates for the associated linear propagator 
to derive an estimate for the space-time norm $\mathcal{B}(T)$.
As in the energy estimates of Section 2, particular attention is paid to the precise dependence of 
these estimates on the non-dimensional parameters $c$ and $\kappa$.

We recall the definitions of the higher-order energy norm
\begin{align*}
	\mathcal{H}_s(T)
	&\coloneqq
	\|
	(c^2-\kappa\Delta)^{\frac{3}{2}}\nabla\phi
	\|_{L^\infty_T H^s_x}
	+
	\|
	(c^2-\kappa\Delta)^2(\lambda\ell)
	\|_{L^\infty_T H^s_x},
	\\
	\mathcal{H}_s(0)
	&\coloneqq
	\|
	(c^2-\kappa\Delta)^{\frac{3}{2}}\nabla\phi_0
	\|_{H^s}
	+
	\|
	(c^2-\kappa\Delta)^2(\lambda\ell_0)
	\|_{H^s},
\end{align*}
and the space-time norm
\begin{equation*}
	\mathcal{B}(T)
	\coloneqq
	\frac{c^2}{\sqrt{\kappa}}\lambda
	\|\ell\|_{L^1_TL^\infty_x}
	+
	c\lambda
	\|\nabla\ell\|_{L^1_TL^\infty_x}
	+
	\sqrt{\kappa}\lambda
	\|\nabla^2\ell\|_{L^1_TL^\infty_x}
	+
	\frac{c}{\sqrt{\kappa}}
	\|\nabla\phi\|_{L^1_TL^\infty_x}
	+
	\|\nabla^2\phi\|_{L^1_TL^\infty_x}.
\end{equation*}
The following lemma provides the estimate for $\mathcal{B}(T)$
needed in the bootstrap argument.

\begin{lem}\label{STEst}
	Let $d \in \mathbb{N}$ satisfy $d \geqslant 2$.
	Assume that the exponents $s\in\mathbb{N}$ and $(q,r)\in \mathbb{R}^2$ satisfy
	\[
	2 \leqslant q \leqslant \infty, \quad 2 \leqslant r < \infty,
	\quad \frac{2}{q} + \frac{d}{r} = \frac{d}{2},
	\quad s > \frac{d}{r}.
	\]
	Let $(\ell_0, \nabla \phi_0) \in H^{s+4}(\mathbb{R}^d) \times H^{s+3}(\mathbb{R}^d)$,
	and let $(\ell, \nabla \phi)$ be a classical solution to the system \eqref{EK_enp}
	in the class $C([0,T]; H^{s+4}(\mathbb{R}^d) \times H^{s+3}(\mathbb{R}^d))
	\cap C^1([0,T]; H^{s+2}(\mathbb{R}^d) \times H^{s+1}(\mathbb{R}^d))$ for some $T>0$.
	Suppose that there exist constants $0<m<1<M$ such that
	\[
	m \leqslant \rho(t,x)
	= \mathcal{L}^{-1}(1+\lambda \ell(t,x)) \leqslant M 
	\]
	for $(t,x) \in [0,T] \times \mathbb{R}^d$.
	Then, there exists a positive constant $C_5=C_5(d,s,q,m,M,K,P)$ such that
	\begin{equation}\label{ST}
		\mathcal{B}(T) \leqslant C_5
		T^{1-\frac{1}{q}} \kappa^{-\frac{q+1}{2q}} c^{-2}
		\left\{
		\mathcal{H}_s(0) + \mathcal{B}(T) \mathcal{H}_s(T)
		\right\}.
	\end{equation}
\end{lem}

\subsection{Integral Formulation}

We first rewrite the system \eqref{EK_enp} in terms of the dispersive
unknown and derive the corresponding integral equations.
We define the differential operators $U, H$ and the dispersive unknown $\psi$ by
\[
U \coloneqq \sqrt{\frac{-\Delta}{c^2-\kappa \Delta}},
\quad
H \coloneqq \sqrt{-\Delta(c^2 - \kappa \Delta)},
\quad \psi \coloneqq U \phi.
\]
Then, the system \eqref{EK_enp} can be transformed into
\begin{equation}\label{EK_ep}
	\begin{cases}
		\partial_t (\lambda \ell) - H\psi 
		= N_1(\ell, \phi),
		\\[2mm]
		\partial_t \psi + H(\lambda \ell)
		= N_2(\ell, \phi),
	\end{cases}
\end{equation}
where the nonlinear terms are defined by
 \begin{align*}
	N_1 = N_1(\ell, \phi)
	&\coloneqq 
	- \nabla\phi \cdot \nabla (\lambda \ell) - \left\{ \tilde{a}(1+\lambda \ell)-1 \right\} \Delta \phi,
	\notag \\
	N_2 = N_2(\ell, \phi)
	&\coloneqq 
	U \left[
	- \dfrac{1}{2}|\nabla \phi|^2
	+ \dfrac{\kappa }{2} |\nabla (\lambda \ell)|^2
	+ \kappa  \left\{
	\tilde{a}(1+\lambda \ell) -1
	\right\} \Delta (\lambda \ell)
	- c^2 
	\left\{
	\tilde{g}(1+\lambda \ell) - \lambda \ell
	\right\}
	\right].
\end{align*}
Note that the linear propagator generated by 
$\mathcal{A}
=\begin{bmatrix}
	0 & H \\
	-H & 0 \\
\end{bmatrix}$ is given by
\[
e^{t\mathcal{A}} = e^{it H} \frac{1}{2}
\begin{bmatrix}
	1&-i\\
	i&1 \\
\end{bmatrix}
+ e^{-it H} \frac{1}{2}
\begin{bmatrix}
	1&i\\
	-i&1
\end{bmatrix}.
\]
Applying the Duhamel formula to \eqref{EK_ep} and using
$\nabla\phi=\nabla U^{-1}\psi$, we obtain the following integral
formulation for $(\lambda\ell,\nabla\phi)$:
\begin{align}\label{IE}
	\begin{bmatrix}
		\lambda \ell(t) \\
		\nabla \phi (t) \\
	\end{bmatrix}
	&= e^{it H}
	\frac{1}{2}
	\begin{bmatrix}
		\lambda \ell_0 - i U \phi_0 \\
		\nabla U^{-1}
		(i \lambda \ell_0 + U \phi_0) \\
	\end{bmatrix}
	+ e^{-itH}
	\frac{1}{2}
	\begin{bmatrix}
		\lambda \ell_0 + i U\phi_0 \\
		\nabla U^{-1}
		(-i \lambda \ell_0 + U \phi_0)
		\\
	\end{bmatrix}
	\notag \\
	&\quad + 
	\int_0^t 
	\left\{
	e^{i(t-\tau) H} 
	\frac{1}{2}
	\begin{bmatrix}
		N_1 - i N_2 \\
		\nabla U^{-1}(i N_1 + N_2) 
	\end{bmatrix}
	+ e^{-i(t-\tau)H} \frac{1}{2}
	\begin{bmatrix}
		N_1 + i N_2 \\
		\nabla U^{-1}(- i N_1 + N_2)
	\end{bmatrix}
	\right\} \, d\tau.
\end{align}

\subsection{Dispersive and Strichartz Estimates for $e^{itH}$}

In this subsection, we derive dispersive and Strichartz estimates
for the linear propagator $e^{itH}$, keeping its dependence on the
non-dimensional parameters $c$ and $\kappa$ explicit.
To this end, we introduce the radial phase 
 \[
 p_\mu(|\xi|) \coloneqq |\xi| \sqrt{1+\mu^2 |\xi|^2}
 \]
for $\mu>0$ and $\xi \in \mathbb{R}^d$.
We first recall the following oscillatory integral estimate.

\begin{lem}\label{osc1}
	Let $\chi \in C^\infty_c([0,\infty);[0,1])$ satisfy $\operatorname{supp} \chi \subset
	[2^{-2}, 2^2]$. Then, it holds that
	\[
	\sup_{x \in \mathbb{R}^d}\left|
	\int_{\mathbb{R}^d} e^{i(x\cdot\xi + t p_\mu(|\xi|))}
	\chi\left( \frac{|\xi|}{2^k} \right) \, d\xi
	\right|
	\lesssim t^{-\frac{d}{2}}
	\mu^{-\frac{d}{2}}
	\left(
	\frac{\mu 2^k}{\sqrt{1+\mu^2 2^{2k}}}
	\right)^{\!\!\frac{d-2}{2}}
	\]
	for $t>0, \, \mu>0$ and $k \in \mathbb{Z}$.
\end{lem}

This estimate follows from the stationary phase estimates given by
\citep[Theorem 2.2]{GNT06} and \citep[Proposition 1]{COX11}.
Indeed, it holds that
\[
|\det \nabla_\xi^2 p_\mu(|\xi|)|^{-\frac{1}{2}}
\sim \mu^{-\frac{d}{2}} 
\left( \frac{\mu |\xi|}{\sqrt{1+\mu^2|\xi|^2}} \right)^{\!\!\frac{d-2}{2}}.
\]
Also, the same estimate was obtained in Subsection 5.3 of \citep{Son24}.

Let 
$\varphi: \mathbb{R}^d \to [0,1]$ be a smooth radial function satisfying
$\operatorname{supp} \varphi \subset \{ 2^{-1} \leqslant |\xi| \leqslant 2 \}$
and $\sum_{k \in \mathbb{Z}}\varphi(2^{-k}\xi) =1$ for $\xi \in \mathbb{R}^d \setminus\{0\}$.
We define the homogeneous Littlewood--Paley localization $P_k$ by
$P_k f \coloneqq \mathcal{F}^{-1}[\varphi(2^{-k} \cdot) \widehat{f}]$ 
for $k \in\mathbb{Z}$ and $f \in \mathscr{S}^\prime(\mathbb{R}^d)$.
Then, it follows from Lemma \ref{osc1} that
\begin{equation*}
	\| e^{it p_\mu(|D|)} P_k f \|_{L^\infty}
	\lesssim t^{-\frac{d}{2}}
	\mu^{-\frac{d}{2}}
	\left(
	\frac{\mu 2^k}{\sqrt{1+\mu^2 2^{2k}}}
	\right)^{\!\!\frac{d-2}{2}}
	\| P_k f \|_{L^1}.
\end{equation*}
Now, since
$H(|\xi|) \coloneqq |\xi| \sqrt{c^2+\kappa|\xi|^2} = c p_{\frac{\sqrt{\kappa}}{c}}(|\xi|)$,
we obtain the dispersive estimate for $e^{itH}$
\begin{equation}\label{osc2}
	\| e^{it H} P_k f \|_{L^\infty}
	\lesssim 
	(\Lambda t)^{-\frac{d}{2}}
	\| P_k f \|_{L^1},
	\qquad 
	\Lambda= \kappa^{\frac{1}{2}}
	\left(
	\frac{\sqrt{\kappa} 2^k}{\sqrt{c^2+ \kappa 2^{2k}}}
	\right)^{\!-\frac{d-2}{d}}.
\end{equation}
Combining \eqref{osc2} with the $L^2$-unitarity of $e^{itH}$, we obtain
the following Strichartz estimates (see \citep{KT98}).

\begin{lem}\label{L_Str1}
	Assume that the exponents $q,r,a$ and $b$ satisfy
	\[
	2 \leqslant q,r,a,b \leqslant \infty,
	\quad
	\frac{2}{q} + \frac{d}{r} = \frac{d}{2},
	\quad
	\frac{2}{a} + \frac{d}{b} = \frac{d}{2},
	\quad
	(d,q,r), (d,a,b)\neq (2,2,\infty).
	\]
	Then, it holds that
	\begin{align*}
		\| e^{i tH} P_k f \|_{L^q_t L^r_x}
		&\lesssim 
		\kappa^{-\frac{1}{2q}}
		\left(
		\frac{\sqrt{\kappa} 2^k}{\sqrt{c^2+ \kappa 2^{2k}}}
		\right)^{\!\frac{d-2}{dq}}
		\| P_k f \|_{L^2_x},
		\\
		\left\|
		\int_0^t e^{i (t-s)H} P_k F(s) \, ds
		\right\|_{L^q_t L^r_x}
		&\lesssim 
		\kappa^{-\frac{1}{2}\left( \frac{1}{q} + \frac{1}{a} \right)}
		\left(
		\frac{\sqrt{\kappa} 2^k}{\sqrt{c^2+ \kappa 2^{2k}}}
		\right)^{\!\frac{d-2}{d}\left( \frac{1}{q} + \frac{1}{a} \right)}
		\| P_k F \|_{L^{a^\prime}_t L^{b^\prime}_x}.
	\end{align*}
\end{lem}

We next rewrite the preceding estimates in a form suitable for use in Subsection 3.3.
By Lemma \ref{L_Str1}, we have
\[
\| e^{itH} f \|_{L^q_t(\dot{B}^0_{r,2})_x}
\lesssim
\kappa^{-\frac{1}{2q}}
\big\|
U_{\kappa,c}^{\frac{d-2}{dq}} f
\big\|_{L^2_x},
\qquad
U_{\kappa,c}
\coloneqq
\frac{\sqrt{\kappa}|D|}
{\sqrt{c^2+\kappa|D|^2}}.
\]
Here,
$
\|f\|_{\dot{B}^0_{r,2}}^2
\coloneqq
\sum_{k\in\mathbb{Z}}\|P_kf\|_{L^r}^2
$
denotes the homogeneous Besov seminorm.
When $d=2$, the multiplier
$U_{\kappa,c}^{(d-2)/dq}$ reduces to the identity.
When $d\geqslant 3$, its symbol is bounded above by one.
Hence, by the embedding 
$
\dot{B}^0_{r,2}(\mathbb{R}^d)
\hookrightarrow L^r(\mathbb{R}^d)
$
for $2\leqslant r <\infty$
and the Minkowski inequality, we have the following estimates.
\begin{cor}\label{L_Str2}
	Assume that the exponents $q$ and $r$ satisfy
	\[
	2\leqslant q\leqslant \infty, \quad 2 \leqslant r < \infty,
	\quad
	\frac{2}{q} + \frac{d}{r} = \frac{d}{2}.
	\]
	Then, there exists a constant $C=C(d,q)>0$ such that
	\begin{align*}
		\| e^{i tH} f \|_{L^q_t L^r_x}
		&\leqslant C
		\kappa^{-\frac{1}{2q}}
		\| f \|_{L^2_x},
		\\
		\left\|
		\int_0^t e^{i (t-s)H} F(s) \, ds
		\right\|_{L^q_t L^r_x}
		&\leqslant C \kappa^{-\frac{1}{2q}}
		\| F \|_{L^1_t L^2_x}.
	\end{align*}
\end{cor}

\subsection{Proof of Lemma \ref{STEst}}

In this subsection, we combine the integral formulation \eqref{IE} with
the Strichartz estimates for $e^{itH}$ to establish the desired estimate \eqref{ST} for $\mathcal{B}(T)$.
Note that $\mathcal{B}(T)$ can be written as
\begin{align*}
	\mathcal{B}(T)
	&\coloneqq  
	\frac{c^2}{\sqrt{\kappa}} \lambda 
	\left\| \ell \right\|_{L^1_T L^\infty_x} 
	+
	c \lambda 
	\| \nabla \ell \|_{L^1_T L^\infty_x}
	+
	\sqrt{\kappa} \lambda 
	\left\| 
	\nabla^2 \ell
	\right\|_{L^1_T L^\infty_x}
	+
	\frac{c}{\sqrt{\kappa}}
	\| \nabla \phi \|_{L^1_T L^\infty_x}
	+
	\| \nabla^2 \phi \|_{L^1_T L^\infty_x}
	\\
	&= \frac{c^2}{\sqrt{\kappa}} 
	\sum_{j=0}^2 \left( \frac{\sqrt{\kappa}}{c} \right)^{j} \| \nabla^j (\lambda \ell) \|_{L^1_T L^\infty_x}
	+ \frac{c}{\sqrt{\kappa}}
	\sum_{j=0}^1 \left( \frac{\sqrt{\kappa}}{c} \right)^{j} \| \nabla^{j} (\nabla \phi) \|_{L^1_T L^\infty_x}.
\end{align*}

\begin{proof}[Proof of Lemma \ref{STEst}]
By the H\"older inequality in time and the Sobolev embedding $W^{s,r}(\mathbb{R}^d) \hookrightarrow L^\infty(\mathbb{R}^d)$,
we see that
 \begin{equation}\label{HS}
 	\mathcal{B}(T)
 	\lesssim 
 	T^{1-\frac{1}{q}}
 	\left\{
 	\frac{c^2}{\sqrt{\kappa}} 
 	\sum_{j=0}^2 \left( \frac{\sqrt{\kappa}}{c} \right)^j \| \nabla^j (\lambda \ell) \|_{L^q_T W^{s,r}_x}
 	+ \frac{c}{\sqrt{\kappa}}
 	\sum_{j=0}^1 \left( \frac{\sqrt{\kappa}}{c} \right)^j \| \nabla^{j} (\nabla \phi) \|_{L^q_T W^{s,r}_x}
 	\right\}.
 \end{equation}
The norms on the right-hand side of \eqref{HS} can now be estimated using the Strichartz estimates.
We first consider the estimates for the linear terms. Set
\begin{align*}
	\lambda \ell^{\text{lin}}(t)
	&\coloneqq 
	e^{it H}
	\frac{1}{2}
	(\lambda \ell_0 - i U \phi_0) 
	+ e^{-itH}
	\frac{1}{2}
	(\lambda \ell_0 + i U \phi_0),
	\\
	\nabla \phi^{\text{lin}}(t)
	&\coloneqq 
	e^{it H}
	\frac{1}{2}
		\nabla U^{-1}
		(i \lambda \ell_0 + U \phi_0) 
	+ e^{-itH}
	\frac{1}{2}
		\nabla U^{-1}
		(-i \lambda \ell_0 + U \phi_0).
\end{align*}
For $f \in \{ \lambda \ell_0, U\phi_0 \}$, it follows from Corollary \ref{L_Str2} that
\begin{align*}
	\frac{c^2}{\sqrt{\kappa}} 
	\sum_{j=0}^2 \left( \frac{\sqrt{\kappa}}{c} \right)^j \| \nabla^j e^{\pm i t H} f \|_{L^q_T W^{s,r}_x}
	&\lesssim 
	\frac{c^2}{\sqrt{\kappa}} 
	\sum_{j=0}^2 \left( \frac{\sqrt{\kappa}}{c} \right)^j 
	\kappa^{-\frac{1}{2q}} \| \nabla^j f \|_{H^{s}_x}
	\\
	&\lesssim
	\kappa^{-\frac{q+1}{2q}} c^{-2}
	\sum_{j=0}^2 
	 \| c^{4-j}  (\sqrt{\kappa}|D|)^j f \|_{H^{s}_x}
	 \\
	 &\lesssim
	 \kappa^{-\frac{q+1}{2q}} c^{-2}
	 \| (c^2 -\kappa \Delta)^2 f \|_{H^s_x}.
\end{align*}
Then, since $\| (c^2 -\kappa \Delta)^2 U \phi_0 \|_{H^s_x}
\sim \| (c^2 -\kappa \Delta)^{\frac{3}{2}} \nabla \phi_0 \|_{H^s_x}$, we have
 \begin{equation}\label{lin_ell}
 	\frac{c^2}{\sqrt{\kappa}} 
 	\sum_{j=0}^2 \left( \frac{\sqrt{\kappa}}{c} \right)^j \| \nabla^j \lambda \ell^{\text{lin}} \|_{L^q_T W^{s,r}_x}
 	\lesssim 
 	\kappa^{-\frac{q+1}{2q}} c^{-2} \mathcal{H}_s(0).
 \end{equation}
Similarly, it holds that for $f \in \{ \lambda \ell_0, U\phi_0 \}$
\begin{align*}
	\frac{c}{\sqrt{\kappa}}
	\sum_{j=0}^1 \left( \frac{\sqrt{\kappa}}{c} \right)^j \| \nabla^{j} e^{\pm it H} (\nabla U^{-1}f) \|_{L^q_T W^{s,r}_x}
	&\lesssim 
	\frac{c}{\sqrt{\kappa}}
	\sum_{j=0}^1 \left( \frac{\sqrt{\kappa}}{c} \right)^j \kappa^{-\frac{1}{2q}} 
	\| \nabla^{j} (\nabla U^{-1}f) \|_{H^{s}_x}
	\\
	&\lesssim 
	\kappa^{-\frac{q+1}{2q}} c^{-2}
	\sum_{j=0}^1 
	\| c^{3-j} (\sqrt{\kappa}|D|)^j \sqrt{c^2-\kappa \Delta} \, f \|_{H^s_x}
	\\
	&\lesssim
	\kappa^{-\frac{q+1}{2q}} c^{-2}
	\| (c^2 -\kappa \Delta)^2 f \|_{H^s_x}.
\end{align*}
This yields
 \begin{equation}\label{lin_phi}
 	\frac{c}{\sqrt{\kappa}}
 	\sum_{j=0}^1 \left( \frac{\sqrt{\kappa}}{c} \right)^j \| \nabla^{j} (\nabla \phi^{\text{lin}}) \|_{L^q_T W^{s,r}_x}
 	\lesssim 
 	\kappa^{-\frac{q+1}{2q}} c^{-2} \mathcal{H}_s(0).
 \end{equation}
Let us next treat the estimates for the nonlinear parts. Put
\begin{align*}
	\lambda \ell^{\text{nl}}(t)
	&\coloneqq \int_0^t 
	\left\{
	e^{i(t-\tau) H} 
	\frac{1}{2}(N_1 - i N_2)
	+ e^{-i(t-\tau)H} \frac{1}{2}
	(N_1 + i N_2)
	\right\} \, d\tau,
	\\
	\nabla \phi^{\text{nl}}(t)
	&\coloneqq 
	\int_0^t 
	\left\{
	e^{i(t-\tau) H} 
	\frac{1}{2}
		\nabla U^{-1}(i N_1 + N_2) 
	+ e^{-i(t-\tau)H} \frac{1}{2}
		\nabla U^{-1}(- i N_1 + N_2)
	\right\} \, d\tau.
\end{align*}
For $k \in \{1, 2\}$, it follows from Corollary \ref{L_Str2} that
\begin{align*}
	\frac{c^2}{\sqrt{\kappa}} 
	\sum_{j=0}^2 \left( \frac{\sqrt{\kappa}}{c} \right)^j 
	\left\| \nabla^j \int_0^t e^{\pm i (t-\tau)H} N_k(\tau) \, d\tau \right\|_{L^q_T W^{s,r}_x}
	&\lesssim 
	\frac{c^2}{\sqrt{\kappa}} 
	\sum_{j=0}^2 \left( \frac{\sqrt{\kappa}}{c} \right)^j  \kappa^{-\frac{1}{2q}}
	\| \nabla^j N_k \|_{L^1_T H^s_x}
	\\
	&\lesssim 
	\kappa^{-\frac{q+1}{2q}} c^{-2}
	\sum_{j=0}^2 c^4 \left( \frac{\sqrt{\kappa}}{c} \right)^j 
	\| \nabla^j N_k \|_{L^1_T H^s_x},
\end{align*}
and
 \begin{align*}
 	&\frac{c}{\sqrt{\kappa}} 
 	\sum_{j=0}^1 \left( \frac{\sqrt{\kappa}}{c} \right)^j 
 	\left\| \nabla^j \int_0^t e^{\pm i (t-\tau)H} \nabla U^{-1} N_k(\tau) \, d\tau \right\|_{L^q_T W^{s,r}_x}
 	\\
 	&\quad \lesssim 
 	\frac{c}{\sqrt{\kappa}} 
 	\sum_{j=0}^1 \left( \frac{\sqrt{\kappa}}{c} \right)^j 
 	\kappa^{-\frac{1}{2q}}
 	\| \nabla^j (\nabla U^{-1} N_k) \|_{L^1_T H^s_x}
 	\\
 	&\quad \lesssim 
 	\kappa^{-\frac{q+1}{2q}} c^{-2}
 	\sum_{j=0}^1 c^3 \left( \frac{\sqrt{\kappa}}{c} \right)^j  
 	\left(
 	c \| \nabla^j N_k \|_{L^1_T H^s_x} + \sqrt{\kappa} \| \nabla^{j+1} N_k \|_{L^1_T H^s_x}
 	\right)
 	\\
 	&\quad \lesssim 
 	\kappa^{-\frac{q+1}{2q}} c^{-2}
 	\left\{
 	\sum_{j=0}^1 c^4 \left( \frac{\sqrt{\kappa}}{c} \right)^j  \| \nabla^j N_k \|_{L^1_T H^s_x}
 	+ \sum_{j=0}^1 c^4 
 	 \left( \frac{\sqrt{\kappa}}{c} \right)^{j+1}  \| \nabla^{j+1} N_k \|_{L^1_T H^s_x}
 	\right\}.
 \end{align*}
Hence it suffices to show that
\[
\sum_{j=0}^2 c^4 \left( \frac{\sqrt{\kappa}}{c} \right)^j 
\left(
\left\| \nabla^j N_1
\right\|_{L^1_T H^{s}_x}
+ \left\| \nabla^j N_2
\right\|_{L^1_T H^{s}_x}
\right)
\lesssim \mathcal{B}(T) \mathcal{H}_s(T).
\]
We first estimate each term in $N_1$.
For the term $\nabla \phi \cdot \nabla (\lambda \ell)$, it follows from Lemma \ref{L_prod} that
\begin{align*}
	&\sum_{j=0}^2 c^4 \left( \frac{\sqrt{\kappa}}{c} \right)^j 
	\| \nabla^j (\nabla \phi \cdot \nabla (\lambda \ell)) \|_{L^1_T H^s_x}
	\\
	&\quad \lesssim 
	\sum_{j=0}^2 c^4 \left( \frac{\sqrt{\kappa}}{c} \right)^j 
	\left(
	\| \nabla (\lambda \ell) \|_{L^1_T L^\infty_x}
	\| \nabla^{j+1} \phi \|_{L^\infty_T H^s_x} 
	+ \| \nabla \phi \|_{L^1_T L^\infty_x}
	\| \nabla^{j+1} (\lambda \ell) \|_{L^\infty_T H^s_x}
	\right)
	\\
	&\quad \lesssim 
	 \sum_{j=0}^2
	\left(
	c \| \nabla (\lambda \ell) \|_{L^1_T L^\infty_x}
	\| c^{3-j} (\sqrt{\kappa}|D|)^j \nabla \phi \|_{L^\infty_T H^s_x} 
	+ \frac{c}{\sqrt{\kappa}}\| \nabla \phi \|_{L^1_T L^\infty_x}
	\| c^{3-j} (\sqrt{\kappa}|D|)^{j+1} (\lambda \ell) \|_{L^\infty_T H^s_x}
	\right)
	\\
	&\quad \lesssim 
	\mathcal{B}(T) \mathcal{H}_s(T).
\end{align*}
Concerning the term $\{\tilde{a}(1+\lambda \ell)-1\} \Delta \phi$,
since $\mathcal{L}(m) \leqslant 1 + \lambda \ell \leqslant \mathcal{L}(M)$ by our assumption,
we have by Lemmas \ref{L_prod} and \ref{L_comp2} that
\begin{align*}
	&\sum_{j=0}^2 c^4 \left( \frac{\sqrt{\kappa}}{c} \right)^j 
	\| \nabla^j 
	\left[
	\left\{ \tilde{a}(1+\lambda \ell)-1 \right\} \Delta \phi
	\right]
	\|_{L^1_T H^s_x}
	\\
	&\quad \lesssim
	\sum_{j=0}^2 c^4 \left( \frac{\sqrt{\kappa}}{c} \right)^j 
	\left(
	\| \nabla^2 \phi \|_{L^1_T L^\infty_x}
	\| \nabla^j \left\{  \tilde{a}(1+\lambda \ell)-1 \right\} \|_{L^\infty_T H^s_x}
	+ \|  \tilde{a}(1+\lambda \ell)-1 \|_{L^1_T L^\infty_x}
	\| \nabla^{j+2} \phi \|_{L^\infty_T H^s_x}
	\right)
	\\
	&\quad \lesssim
	\sum_{j=0}^2 c^4 \left( \frac{\sqrt{\kappa}}{c} \right)^j 
	\left(
	\| \nabla^2 \phi \|_{L^1_T L^\infty_x}
	\| \nabla^j (\lambda \ell) \|_{L^\infty_T H^s_x}
	+ \|  \lambda \ell \|_{L^1_T L^\infty_x}
	\| \nabla^{j+2} \phi \|_{L^\infty_T H^s_x}
	\right)
	\\
	&\quad \lesssim
	\sum_{j=0}^2 
	\left(
	\| \nabla^2 \phi \|_{L^1_T L^\infty_x}
	\| c^{4-j} (\sqrt{\kappa}|D|)^j (\lambda \ell) \|_{L^\infty_T H^s_x}
	+ \frac{c^2}{\sqrt{\kappa}} \|  \lambda \ell \|_{L^1_T L^\infty_x}
	\| c^{2-j} (\sqrt{\kappa}|D|)^{j+1} \nabla \phi \|_{L^\infty_T H^s_x}
	\right)
	\\
	&\quad \lesssim 
	\mathcal{B}(T) \mathcal{H}_s(T).
\end{align*}
We next derive bounds for each term in $N_2$. Note that
\[
\| U f \|_{L^2_x} = \frac{1}{\sqrt{\kappa}} \left\| \frac{\sqrt{\kappa}|D|}{\sqrt{c^2- \kappa \Delta}} f\right\|_{L^2_x}
\leqslant \frac{1}{\sqrt{\kappa}} \| f \|_{L^2_x}.
\]
Thus, in the $L^2$-estimate, the operator $U$ provides a gain of $\kappa^{-1/2}$.
For the term $U|\nabla\phi|^2$, it follows from Lemma \ref{L_prod} that
\begin{align*}
	\sum_{j=0}^2 c^4 \left( \frac{\sqrt{\kappa}}{c} \right)^j 
	\| \nabla^j U [|\nabla\phi|^2]
	\|_{L^1_T H^s_x}
	&\leqslant 
	\sum_{j=0}^2 c^4 \left( \frac{\sqrt{\kappa}}{c} \right)^j \frac{1}{\sqrt{\kappa}}
	\| \nabla^j [|\nabla\phi|^2]
	\|_{L^1_T H^s_x}
	\\
	&\lesssim
	\sum_{j=0}^2 c^4 \left( \frac{\sqrt{\kappa}}{c} \right)^j \frac{1}{\sqrt{\kappa}}
	\| \nabla \phi \|_{L^1_T L^\infty_x}
	\| \nabla^{j+1} \phi \|_{L^\infty_T H^s_x}
	\\
	&\lesssim 
	\sum_{j=0}^2
	\frac{c}{\sqrt{\kappa}} \| \nabla \phi \|_{L^1_T L^\infty_x}
	\| c^{3-j} (\sqrt{\kappa}|D|)^{j} \nabla \phi \|_{L^\infty_T H^s_x}
	\\
	&\lesssim 
	\mathcal{B}(T) \mathcal{H}_s(T).
\end{align*}
Similarly, for the term $U[\kappa|\nabla(\lambda \ell)|^2]$, we have
\begin{align*}
	\sum_{j=0}^2 c^4 \left( \frac{\sqrt{\kappa}}{c} \right)^j 
	\| \nabla^j U [\kappa |\nabla (\lambda \ell)|^2]
	\|_{L^1_T H^s_x}
	&\lesssim
	\sum_{j=0}^2 c^4 \left( \frac{\sqrt{\kappa}}{c} \right)^j 
	\sqrt{\kappa}
	\| \nabla (\lambda \ell) \|_{L^1_T L^\infty_x}
	\| \nabla^{j+1} (\lambda \ell) \|_{L^\infty_T H^s_x}
	\\
	&\lesssim  
	\sum_{j=0}^2
	c \| \nabla (\lambda \ell) \|_{L^1_T L^\infty_x}
	\| c^{3-j} (\sqrt{\kappa}|D|)^{j+1} (\lambda \ell) \|_{L^\infty_T H^s_x}
	\\
	&\lesssim 
	\mathcal{B}(T) \mathcal{H}_s(T).
\end{align*}
Concerning the term $U [\kappa  \left\{
\tilde{a}(1+\lambda \ell) -1
\right\} \Delta (\lambda \ell)]$,
Lemmas \ref{L_prod}, \ref{L_comp2}
and the bound $\mathcal{L}(m) \leqslant 1 + \lambda \ell \leqslant \mathcal{L}(M)$ yield
\begin{align*}
	&\sum_{j=0}^2 c^4 \left( \frac{\sqrt{\kappa}}{c} \right)^j 
	\| \nabla^j U [\kappa  \left\{
	\tilde{a}(1+\lambda \ell) -1
	\right\} \Delta (\lambda \ell)]
	\|_{L^1_T H^s_x}
	\\
	&\quad \lesssim
	\sum_{j=0}^2 c^4 \left( \frac{\sqrt{\kappa}}{c} \right)^j 
	\sqrt{\kappa}
	\left(
	\| \lambda \ell \|_{L^1_T L^\infty_x}
	\| \nabla^{j+2} (\lambda \ell) \|_{L^\infty_T H^s_x}
	+ \| \nabla^2 (\lambda \ell) \|_{L^1_T L^\infty_x}
	\| \nabla^j (\lambda\ell) \|_{L^\infty_T H^s_x}
	\right)
	\\
	&\quad \lesssim 
	\sum_{j=0}^2
	\left(
	\frac{c^2}{\sqrt{\kappa}} \| \lambda \ell \|_{L^1_T L^\infty_x}
	\| c^{2-j} (\sqrt{\kappa}|D|)^{j+2} (\lambda \ell) \|_{L^\infty_T H^s_x}
	+ \sqrt{\kappa} \| \nabla^2 (\lambda \ell) \|_{L^1_T L^\infty_x}
	\| c^{4-j} (\sqrt{\kappa}|D|)^j (\lambda\ell) \|_{L^\infty_T H^s_x}
	\right)
	\\
	&\quad \lesssim
	\mathcal{B}(T) \mathcal{H}_s(T).
\end{align*}
We finally treat the pressure term 
$U [
c^2 
\left\{
\tilde{g}(1+\lambda \ell) - \lambda \ell
\right\}
]$.
Since $\tilde{g}(1)=0$ and $\tilde{g}^\prime(1)=1$, 
it follows from the Taylor theorem, Lemmas \ref{L_prod}, \ref{L_comp1} and
$\mathcal{L}(m) \leqslant 1 + \lambda \ell \leqslant \mathcal{L}(M)$ that
\begin{align*}
	\sum_{j=0}^2 c^4 \left( \frac{\sqrt{\kappa}}{c} \right)^j 
	\| \nabla^j U [
	c^2 
	\left\{
	\tilde{g}(1+\lambda \ell) - \lambda \ell
	\right\}
	]
	\|_{L^1_T H^s_x}
	&\lesssim
	\sum_{j=0}^2 \frac{c^6}{\sqrt{\kappa}} \left( \frac{\sqrt{\kappa}}{c} \right)^j
	\| \lambda \ell \|_{L^1_T L^\infty_x}
	\| \nabla^j (\lambda \ell) \|_{L^\infty_T H^s_x}
	\\
	&\lesssim \sum_{j=0}^2
	\frac{c^2}{\sqrt{\kappa}} \| \lambda \ell \|_{L^1_T L^\infty_x}
	\| c^{4-j} (\sqrt{\kappa}|D|)^j (\lambda \ell) \|_{L^\infty_T H^s_x}
	\\
	&\lesssim 
	\mathcal{B}(T) \mathcal{H}_s(T).
\end{align*}
Therefore, we obtain
\begin{equation}\label{nl_ellphi}
	\frac{c^2}{\sqrt{\kappa}} 
	\sum_{j=0}^2 \left( \frac{\sqrt{\kappa}}{c} \right)^{j} \| \nabla^j (\lambda \ell^{\text{nl}}) \|_{L^q_T W^{s,r}_x}
	+ \frac{c}{\sqrt{\kappa}}
	\sum_{j=0}^1 \left( \frac{\sqrt{\kappa}}{c} \right)^{j} \| \nabla^{j} (\nabla \phi^{\text{nl}}) \|_{L^q_T W^{s,r}_x}
	\lesssim 
	\kappa^{-\frac{q+1}{2q}} c^{-2}
	\mathcal{B}(T) \mathcal{H}_s(T),
\end{equation}
where the implicit constant depends on $d,s,q,m,M,K$ and $P$.
Hence the desired estimate follows from \eqref{HS}, \eqref{lin_ell}, \eqref{lin_phi} and \eqref{nl_ellphi}.
\end{proof}

\section{Proof of Theorem \ref{main}}

We are now ready to present the proof of Theorem \ref{main}.

\begin{proof}[Proof of Theorem \ref{main}]
Let $d \geqslant 2$ be an integer.
Let $q\in(2,\infty)$ if $d=2$, and let $q\in[2,\infty)$ if $d\geqslant 3$.
Take $s\in\mathbb{N}$ with $s>d/2$.
Suppose that $(\ell_0, \nabla \phi_0) \in H^{s+4}(\mathbb{R}^d) \times H^{s+3}(\mathbb{R}^d)$ satisfies
\begin{equation}\label{small2}
	\mathcal{E}_s(0) = 
	\|\nabla\phi_0\|_{H^{s}} + 
	\| \sqrt{c^2-\kappa \Delta} \, (\lambda \ell_0) \|_{H^{s}} \leqslant \varepsilon c
\end{equation}
for $\varepsilon>0$ to be chosen later. We take the exponent $r \in [2,\infty)$ so that
$(q,r)$ satisfies
the admissible condition  $\frac{2}{q}+\frac{d}{r}=\frac{d}{2}$.
Note that $s>d/2 \geqslant d/r$.

We first verify that the system \eqref{EK_enp} admits a local-in-time solution.
By the Sobolev embedding $H^s(\mathbb{R}^d) \hookrightarrow L^\infty(\mathbb{R}^d)$
and \eqref{small2}, we see that
\[
\| \lambda \ell_0 \|_{L^\infty}
\leqslant C_{d,s} \| \lambda \ell_0 \|_{H^s}
\leqslant \frac{C_{d,s}}{c} \| \sqrt{c^2-\kappa \Delta} \, (\lambda\ell_0) \|_{H^s}
\leqslant C_{d,s} \varepsilon.
\]
Here, $C_{d,s}>0$ denotes the constant in the Sobolev inequality
$\| f \|_{L^\infty} \leqslant C_{d,s} \| f \|_{H^s}$.
We first require $\varepsilon>0$ to satisfy
\begin{equation}\label{eps1}
	0< \varepsilon \leqslant \frac{1}{C_{d,s}}\min\left\{
	1-\mathcal{L}\left( \frac{3}{4} \right), \, \mathcal{L}\left( \frac{5}{4} \right) -1
	\right\}.
\end{equation}
Then, it holds that $\mathcal{L}(3/4) \leqslant 1 + \lambda \ell_0(x) \leqslant \mathcal{L}(5/4)$,
which is equivalent to $3/4 \leqslant \rho_0(x) \leqslant 5/4$ for $x \in \mathbb{R}^d$.
Also, since $\rho_0 -1 = \mathcal{L}^{-1}(1+\lambda \ell_0)-1$,
Lemma \ref{L_comp1} implies $\rho_0-1 \in H^{s+4}(\mathbb{R}^d)$.
Therefore, the local existence theorem established in \citep{BGDD07} yields
$T_0>0$ and a unique local solution to the Euler--Korteweg system
\eqref{EK_nond} such that
\[
(\rho-1, u) \in 
C([0, T_0];H^{s+4}(\mathbb{R}^d) \times H^{s+3}(\mathbb{R}^d)) \cap
C^1([0, T_0]; H^{s+2}(\mathbb{R}^d) \times H^{s+1}(\mathbb{R}^d))
\]
and $1/2 \leqslant \rho(t,x) \leqslant 3/2$ for $(t,x) \in [0, T_0] \times \mathbb{R}^d$.
Also, $\operatorname{curl} u_0 =0$ implies $\operatorname{curl} u(t,x)=0$.
Then, since $\lambda \ell = \mathcal{L}(1+(\rho-1))-1$, we see from Lemma \ref{L_comp1} that
\begin{equation}\label{solc}
	(\ell, \nabla\phi ) \in 
	C([0, T_0];H^{s+4}(\mathbb{R}^d) \times H^{s+3}(\mathbb{R}^d)) \cap
	C^1([0, T_0]; H^{s+2}(\mathbb{R}^d) \times H^{s+1}(\mathbb{R}^d))
\end{equation}
is a unique solution to \eqref{EK_enp} on $[0, T_0]$ satisfying
$\mathcal{L}\left( 1/2 \right) \leqslant 
1 + \lambda \ell(t,x)
\leqslant \mathcal{L}\left( 3/2 \right)$.

Now, let $T>0$ be a time to be specified later, 
and let $T^\ast$ be the maximal existence time of the solution $(\ell, \nabla \phi)$
in the class \eqref{solc}.
We define
\begin{equation*}
	T^\prime
	\coloneqq \sup\left\{
	t \in [0,T] \cap [0, T^\ast) 
	\bigm| 
	\mathcal{E}_{s}(t) \leqslant 2C_1 e^{C_2} \mathcal{E}_s(0),
	\ \ 
	\mathcal{H}_s(t) \leqslant 2 C_3  e^{C_4} \mathcal{H}_s(0),
	\ \  
	\mathcal{B}(t) \leqslant 1
	\right\}.
\end{equation*}
Here, $C_j \ (j=1,2,3,4)$ denote the constants appearing in Lemma \ref{EnergyEst}.
Enlarging $C_1$ and $C_3$ if necessary, we may assume that $C_1,C_3\geqslant 1$.
Since $(\ell, \nabla\phi )$ is continuous in time with values in 
$H^{s+4}(\mathbb{R}^d) \times H^{s+3}(\mathbb{R}^d)$,
we see that $T^\prime>0$.

We claim that
$T^\prime=\min\{T,T^\ast\}$
for a suitable choice of $T>0$ as specified below.
Suppose, to the contrary, that
\[
T^\prime<\min\{T,T^\ast\}.
\]
The Sobolev embedding $H^s(\mathbb{R}^d) \hookrightarrow L^\infty(\mathbb{R}^d)$
and the smallness condition on initial data \eqref{small2} imply
that for $0 \leqslant t \leqslant T^\prime$
\begin{align*}
	\lambda \| \ell(t) \|_{L^\infty_x}
	&\leqslant C_{d,s}  \| \lambda \ell(t) \|_{H^{s}_x}
	\leqslant \frac{C_{d,s}}{c} \| \sqrt{c^2 -\kappa \Delta} \lambda \ell(t) \|_{H^{s}_x}
	\\
	&\leqslant \frac{C_{d,s}}{c} \mathcal{E}_{s}(T^\prime)
	\leqslant
	\frac{2C_{d,s} C_1 e^{C_2}}{c} \mathcal{E}_s(0)
	\\
	&\leqslant 2C_{d,s} C_1 e^{C_2} \varepsilon.
\end{align*}
Here, we further require $\varepsilon$ to satisfy
 \begin{equation}\label{eps2}
 	0< \varepsilon \leqslant \frac{1}{2C_{d,s} C_1 e^{C_2}}\min\left\{
 	1-\mathcal{L}\left( \frac{3}{4} \right), \, \mathcal{L}\left( \frac{5}{4} \right) -1
 	\right\}.
 \end{equation}
 Then it follows that
 \begin{equation}\label{Linfell}
 	\mathcal{L}\left(\frac{3}{4}\right)
 	\leqslant
 	1+\lambda\ell(t,x)
 	\leqslant
 	\mathcal{L}\left(\frac{5}{4}\right)
 \end{equation}
for all $(t,x)\in[0,T^\prime]\times\mathbb{R}^d$.
Therefore, we can apply the energy estimates \eqref{EE1} and \eqref{EE2} in Lemma \ref{EnergyEst}
and the space-time estimate \eqref{ST} in Lemma \ref{STEst}.

By the space-time estimate \eqref{ST} and the definition of $T^\prime$, we see that
\begin{align*}
	\mathcal{B}(T^\prime) 
	&\leqslant C_5 
	(T^\prime)^{1-\frac{1}{q}}\kappa^{-\frac{1}{2q}-\frac{1}{2}} c^{-2}
	\left\{
	\mathcal{H}_s(0) + \mathcal{B}(T^\prime) \mathcal{H}_s(T^\prime)
	\right\}
	\\
	&\leqslant C_5
	T^{1-\frac{1}{q}}\kappa^{-\frac{1}{2q}-\frac{1}{2}} c^{-2}
	\left\{
	\mathcal{H}_s(0) + 2 C_3 e^{C_4} \mathcal{H}_s(0)
	\right\}
	\\
	&= C_5 (1+ 2C_3 e^{C_4})
	T^{1-\frac{1}{q}}\kappa^{-\frac{q+1}{2q}} c^{-2} \mathcal{H}_s(0).
\end{align*}
We now choose $T>0$ so that
\begin{equation}\label{time}
	T \leqslant 
	\left\{
	\frac{1}{2 C_5(1+ 2C_3 e^{C_4})}
	\right\}^{\!\frac{q}{q-1}}
	\kappa^{\frac{q+1}{2(q-1)}}
	c^{\frac{2q}{q-1}}
	\mathcal{H}_s(0)^{-\frac{q}{q-1}}.
\end{equation}
This yields $\mathcal{B}(T^\prime) \leqslant 1/2$. Moreover, it follows that
\begin{align*}
	\lambda \int_0^{T^\prime} \| \nabla \phi(t) \|_{L^\infty} \| \nabla \ell(t) \|_{L^\infty} \, dt
	&\leqslant
	C_{d,s}  \int_0^{T^\prime} \lambda
	\| \nabla \phi(t) \|_{H^{s}} 
	\| \nabla  \ell(t) \|_{L^\infty} \, dt
	\\
	&\leqslant \frac{C_{d,s}}{c} \mathcal{E}_{s}(T^\prime)  c \lambda 
	\| \nabla \ell \|_{L^1_{T^\prime}L^\infty_x} 
	\\
	&\leqslant 
	\frac{C_{d,s}}{c} 2 C_1 e^{C_2} \mathcal{E}_{s}(0)\mathcal{B}(T^\prime)
	\\
	&\leqslant C_{d,s} C_1 e^{C_2} \varepsilon.
\end{align*}
Finally, we require $\varepsilon>0$ to satisfy
 \begin{equation}\label{eps3}
 	0 < \varepsilon \leqslant \frac{1}{4 C_{d,s}C_1 e^{C_2}}.
 \end{equation}
 We take $\varepsilon>0$ satisfying
 \eqref{eps1}, \eqref{eps2}, and \eqref{eps3}.
 Then, we have 
 \[
 \lambda \int_0^{T^\prime} \| \nabla \phi(t) \|_{L^\infty} \| \nabla \ell(t) \|_{L^\infty} \, dt
 \leqslant \frac{1}{4}.
 \]
Therefore, 
 it follows from the energy estimates \eqref{EE1} for $\mathcal{E}_s(t)$ and 
 \eqref{EE2} for $\mathcal{H}_s(t)$ that
 \begin{align*}
 	\mathcal{E}_s(T^\prime )
 	&\leqslant C_1 \mathcal{E}_s(0)
 	\exp\left\{ C_2 \left(
 	\mathcal{B}(T^\prime) + \lambda \int_0^{T^\prime} \| \nabla \phi(\tau) \|_{L^\infty} \| \nabla \ell(\tau) \|_{L^\infty} \, d\tau
 	\right) \right\}
 	\\
 	&\leqslant C_1 \mathcal{E}_s(0) \exp\left\{
 	\frac{3}{4} C_2
 	\right\},
 \end{align*}
and
 \begin{align*}
 	\mathcal{H}_s(T^\prime) 
 	&\leqslant C_3 \mathcal{H}_s(0)
 	\exp\left\{ C_4 \left(
 	\mathcal{B}(T^\prime) + \lambda \int_0^{T^\prime} \| \nabla \phi(\tau) \|_{L^\infty} \| \nabla \ell(\tau) \|_{L^\infty} \, d\tau
 	\right) \right\}
 	\\
 	&\leqslant C_3 
 	\mathcal{H}_s(0)
 	\exp\left\{
 	\frac{3}{4}C_4
 	\right\},
 \end{align*}
respectively. Hence we obtain
\[
\mathcal{E}_{s}(T^\prime) \leqslant C_1 e^{\frac{3}{4}C_2} \mathcal{E}_s(0),
\quad 
\mathcal{H}_s(T^\prime) \leqslant C_3 e^{\frac{3}{4}C_4} \mathcal{H}_s(0),
\quad 
\mathcal{B}(T^\prime) \leqslant \frac{1}{2}
\]
provided that the time $T$ satisfies \eqref{time}.
These estimates are strict improvements of the bootstrap assumptions in the definition of $T^\prime$.
Since $T^\prime<\min\{T,T^\ast\}$, the time continuity of the solution $(\ell, \nabla\phi)$ implies that
there exists $T^\prime < \tilde{T} < \min\{ T, \, T^\ast \}$ such that
\[
\mathcal{E}_{s}(\tilde{T}) \leqslant \frac{3}{2}C_1 e^{C_2} \mathcal{E}_s(0),
\quad 
\mathcal{H}_s(\tilde{T}) \leqslant \frac{3}{2}C_3 e^{C_4} \mathcal{H}_s(0),
\quad
\mathcal{B}(\tilde{T}) \leqslant \frac{3}{4}.
\]
This contradicts the definition of $T^\prime$.
Hence we have $T^\prime= \min\{ T, \, T^\ast \}$.

It remains to show that $T^\ast\geqslant T$.
Suppose, to the contrary, that $T^\ast<T$. Then we have $T^\prime=T^\ast$.
Let $S \in (0, T^\ast)$. Similarly to \eqref{Linfell}, the smallness assumption on $\mathcal{E}_s(0)$
and \eqref{eps2} yields $3/4 \leqslant \rho(t,x) \leqslant 5/4$ for $(t,x) \in [0, S] \times \mathbb{R}^d$.
Since $S<T^\ast$ is arbitrary, it holds that 
\[
\rho([0, T^\ast) \times \mathbb{R}^d) \subset \left[\frac{3}{4}, \, \frac{5}{4}\right].
\]
Moreover, since $u=\nabla \phi$, we have
\[
\int_0^S \| \operatorname{div} u(t) \|_{L^\infty} \, dt
= \int_0^S \| \Delta \phi(t) \|_{L^\infty} \, dt
\leqslant \mathcal{B}(S) \leqslant 1.
\]
We next control $\| \Delta\rho \|_{L^1_S L^\infty_x}$. Since
 \[
 \Delta \rho 
 = \frac{1}{\mathcal{L}^\prime(\rho)} \lambda \Delta \ell
 - \frac{\mathcal{L}^{\prime\prime}(\rho)}{\mathcal{L}^\prime(\rho)^3} \lambda^2 |\nabla \ell|^2
 \]
and $3/4 \leqslant \rho \leqslant 5/4$, it holds that
\begin{align*}
	\int_0^S \| \Delta \rho(t) \|_{L^\infty} \, dt
	&\lesssim 
	\lambda \int_0^S \| \Delta \ell(t) \|_{L^\infty} \, dt
	+ \lambda^2 \int_0^S \| \nabla \ell(t) \|_{L^\infty}^2 \, dt
	\\
	&\lesssim \frac{1}{\sqrt{\kappa}} \mathcal{B}(S)
	+ \frac{1}{c\sqrt{\kappa}} \mathcal{E}_s(S) \mathcal{B}(S)
	\\
	&\lesssim 
	\left\{
	\frac{1}{\sqrt{\kappa}} + \frac{C_1 e^{C_2}}{c\sqrt{\kappa}} \mathcal{E}_s(0)
	\right\}.
\end{align*}
Together with $\operatorname{curl}u=0$, we obtain
\[
\int_0^{T^\ast}
\left(
\|\operatorname{div}u(t)\|_{L^\infty}
+
\|\Delta\rho(t)\|_{L^\infty}
\right)\,dt
<\infty.
\]
Therefore, the continuation criterion for the Euler--Korteweg system established
in \citep{BGDD07} implies that $(\rho,u)$ can be extended beyond $T^\ast$.
Equivalently, $(\ell,\nabla\phi)$ can be continued beyond $T^\ast$,
contradicting the maximality of $T^\ast$.
Therefore, $T^\ast\geqslant T$, and hence
\[
T^\ast
\geqslant
\delta
\kappa^{\frac{q+1}{2(q-1)}}
c^{\frac{2q}{q-1}}
\mathcal{H}_s(0)^{-\frac{q}{q-1}}
\]
for some $\delta=\delta(d,s,q,K,P)>0$.
This completes the proof of Theorem \ref{main}.
\end{proof}

\section*{Acknowledgments}
The authors would like to thank Professor Hirokazu Saito
for valuable discussions and helpful advice.
R. Takada was supported in part by JSPS KAKENHI Grant Numbers JP22K03388 and JP26K06876.
T. Yoshizawa was supported in part by FoPM, WINGS Program, the University of Tokyo.

\bibliographystyle{alpha} 
\bibliography{refs} 

\end{document}